\documentclass{article}
\usepackage{authblk}
\usepackage[utf8]{inputenc}
\usepackage{url}
\usepackage[version=4]{mhchem}
\usepackage{multirow}
\usepackage{lipsum}
\usepackage{stackrel}
\usepackage{listings}
\usepackage{amssymb}
\usepackage{nccmath}
\usepackage{a4wide}
\usepackage{mathtools}
\usepackage{bbm}
\usepackage{amsmath,amsfonts,amsthm,amssymb}
\usepackage{bbm, dsfont}
\usepackage[justification=justified, singlelinecheck=false]{caption}
\usepackage{graphicx}
\usepackage{enumitem}
\usepackage{float}

\usepackage[justification=centering]{caption}
\usepackage{algorithmicx}
\usepackage{enumitem}
\usepackage{empheq,etoolbox}
\patchcmd{\subequations}
  {\theparentequation\alph{equation}}
  {\theparentequation.\alph{equation}}
  {}{}
\usepackage[justification=centering]{caption}
\usepackage{algorithmicx}
\usepackage{algorithm}
\usepackage{algpseudocode}
\usepackage{stackengine}

\usepackage[title]{appendix}
\usepackage[scaled=.90]{helvet}
\usepackage{courier}
\usepackage{ae}
\usepackage[T1]{fontenc}
\usepackage{graphicx}
\usepackage{subcaption}
\usepackage[export]{adjustbox}
\usepackage{wrapfig}

\usepackage[T1]{fontenc}
\usepackage{here} 
\usepackage[colorlinks=false, pdfborder={0 0 0}]{hyperref}

\usepackage{mdwlist}
\usepackage{tcolorbox}
\usepackage{fancybox}
\usepackage{framed}

\definecolor{darkblue}{rgb}{0,0,0.8}
\definecolor{darkgreen}{rgb}{0,0.8,0}

\definecolor{magenta}{rgb}{0.5,0,0.5}

\newcommand{\mathleft}{\@fleqntrue\@mathmargin0pt}

\newtheorem{theorem}{Theorem}[section]

\newtheorem{assumption}{Assumption}[section]

\newtheorem{proposition}{Proposition}[section]
\newtheorem{remark}{Remark}[section]

\newtheorem{notation}{Notation}[section]

\providecommand{\keywords}[1]
{
  \small	
  \textbf{\textit{Keywords---}} #1
}

\usepackage{natbib}
\begin{document}
\title{First-passage time for PDifMPs: an Exact simulation approach for time-varying thresholds}
\author[1]{Sascha Desmettre\thanks{Email: sascha.desmettre@jku.at}}
\author[2]{Devika Khurana\thanks{devika.khurana@jku.at}}
\author[3]{Amira Meddah\thanks{amira.meddah@jku.at}}
\affil[1]{\centerline{\small Institute of Financial Mathematics and Applied Number Theory, Johannes Kepler University, Linz}}
\affil[2]{\centerline{\small Institute of Numerical Mathematics, Johannes Kepler University, Linz}}
\affil[3]{\centerline{\small Institute of Stochastics, Johannes Kepler University, Linz}}
\date{\today}

\newcommand{\dkbox}[1]{
\fbox{
    \parbox{0.9\textwidth}{%
        $\triangleright$\textcolor{brown}{\textbf{Devika}:} #1
        }
 }}

\newcommand{\sdbox}[1]{\fbox{$\triangleright$\textcolor{blue}{\textbf{Sascha}:} #1}}
\newcommand{\amebbox}[1]{{
\fbox{
\parbox{0.9\textwidth}{  \fbox{$\triangleright$\asd{\textbf{Sascha}:}} 
#1
}}}}


\maketitle
\begin{abstract}
    Piecewise Diffusion Markov Processes (PDifMPs) are valuable for modelling systems where continuous dynamics are  interrupted by sudden shifts and/or changes in drift and diffusion. The first-passage time (FPT) in such models plays a central role in understanding when a process first reaches a critical boundary. In many systems, time-dependent thresholds provide a flexible framework for reflecting evolving conditions, making them essential for realistic modelling.
    We propose a hybrid exact simulation scheme for computing the FPT of PDifMPs to time-dependent thresholds. Exact methods traditionally exist for pure diffusions, using Brownian motion as an auxiliary process and accepting sampled paths with a probability weight. Between jumps, the PDifMP evolves as a diffusion, allowing us to apply the exact method within each inter-jump interval. The main challenge arises when no threshold crossing is detected in an interval: We then need the value of the process at the jump time, and for that, we introduce an approach to simulate a conditionally constrained auxiliary process and derive the corresponding acceptance probability. Furthermore, we prove the convergence of the method and illustrate it using numerical examples.
\end{abstract}

\section{Introduction}

The first-passage time of a stochastic process represents the earliest time instance at which the process reaches a specific threshold for the first time. It is a fundamental concept widely used to describe critical events in various applications. The significance of FPTs lies in their ability to capture and characterise meaningful transitions or decisive events. For instance, in neuroscience and in biology, FPTs characterise the timing of essential neuronal activities, such as when a neuron's membrane potential reaches or surpasses a firing threshold, signalling events like the initiation of neural spikes or bursts of activity, e.g., \cite{braun2015first, engel2007firing, lo2006first, chou2014first}. In finance, they are essential for example in pricing barrier options and managing risk, marking the precise time when an asset price reaches a crucial market level that triggers investment decisions, we refer for example to \cite{hieber2012note, fernandez2013double, kammer2007general}.\\
Exact simulation method for first-passage time aim to generate samples from the true distribution without introducing time discretisation errors. Unlike numerical schemes that approximate trajectories on fixed time grids, Exact methods typically rely on probabilistic constructions or change-of-measure technique that allow one to simulate hitting times directly. The development of these methods began with foundational work by Beskos et al., who introduced a rejection sampling framework based on Girsanov’s transformation to simulate finite points from the path of a continuous stochastic differential equation (SDE). The core idea was to use Brownian motion as an auxiliary process and accept or reject its trajectory as a valid realisation of the diffusion process, with a probability computed via Girsanov’s transformation, cf. \cite{beskos2005exact}.

This idea was later adapted to directly simulate the FPT of a diffusion to a constant threshold, using a version of Girsanov’s transformation formulated for stopping times, cf. \cite{herrmann2019exact}. Extending the method to time-dependent thresholds presented additional challenges, as obtaining the required probability from Girsanov’s transformation was not straightforward. These challenges were addressed in \cite{khurana2024exact}. The method is advantageous for simulating FPTs as it eliminates the need to simulate the entire path of the diffusion, avoids time discretisation errors, and does not rely on specific time intervals.\\
While the Exact simulation method is well established for pure continuous diffusion, many real-world systems exhibit abrupt changes, motivating the study of jump diffusion processes. Jump diffusion processes were first introduced in \cite{merton1976option} by Merton in 1976 to address limitations in modelling financial asset prices using pure diffusion models. Merton recognised that standard Brownian motion could not capture sudden, significant changes observed in financial markets. To overcome this, the author proposed a model that combined continuous Brownian motion with jumps governed by a Poisson process, allowing for a more accurate representation of asset price dynamics. Since Merton's pioneering work, research on jump diffusion processes has expanded considerably, finding applications in diverse fields such as finance, physics, biology and engineering \cite{jang2007jump, giraudo1997jump}.  A major challenge in this context is the computation of first-passage times, which are critical in applications such as credit risk, neuroscience, and failure prediction. Recently, Herrmann and Massin in \cite{herrmann2023exact} proposed an algorithm that exactly simulates the FPT through a constant threshold in one-dimensional jump diffusion processes, that uses the classical Exact simulation method developed for continuous diffusion. Their approach reconstructs the jump diffusion by treating the jump times and jump sizes separately from the continuous dynamics between jumps. The classical Exact simulation method is used to compute the FPT for the continuous part. If this FPT occurs before the next jump, it is also the FPT of the full process. Otherwise, the value of the continuous part at the jump time is simulated, the jump size is added, and the continuous process is reinitialised from this new point to continue the FPT computation. This step implicitly relies on simulating the continuous process conditioned to remain below the threshold until the next jump time. For constant thresholds, the distribution of the auxiliary process—a Brownian motion conditioned in this way, is known. And thus one can perform rejection sampling.\\
Analytical expressions for the FPT density of jump diffusions are generally unavailable, except when the jump size follows a doubly exponential distribution \cite{kou2003first} or is non-negative with the boundary below the initial value \cite{blake2003level}. The main challenge lies in the possibility of an overshoot when the process crosses the boundary. While the overshoot distribution is tractable for exponential jumps due to the memoryless property, it becomes difficult to handle for more general jump size distributions. When analytical results are unavailable, simulation-based approaches are often used. One possibility is to simulate full sample paths using numerical methods such as Runge-Kutta or Euler-type schemes \cite{buckwar2011runge, higham2005numerical}, and compute the FPT as the by product. Another direction avoids modelling jumps explicitly by considering a piecewise-defined threshold instead; in \cite{abundo2010first}, the author shows that the hitting time in this setting satisfies a certain integral equation, which is then solved numerically. A different approach, introduced in \cite{atiya2005efficient}, proposes two Monte Carlo methods for Brownian motion. A key distinction of this method is that it requires only a few simulation points per iteration and avoids discretisation bias.\\
FPT problems for jump diffusions with time-varying thresholds have also been explored in \cite{abundo2000first, tuckwell1976first}.\\
As the need to model more complex systems grew, jump diffusion models were included into the broader framework of general stochastic hybrid systems (GSHS), also called piecewise diffusion Markov processes (PDifMPs). These systems integrate stochastic differential equations with discrete events, allowing rigorous analysis of systems where random jumps, mode switches or resets interrupt continuous evolution. PDifMPs generalise jump diffusions by allowing the drift and diffusion coefficients of the continuous part to depend on the discrete jump component, thus reflecting the effects of jumps more broadly. PDifMPs have proven particularly useful in domains where such hybrid behaviours occur naturally, including but not limited to biology, e.g. \cite{buckwar2025numerical, meddah2023stochastic, debussche2018modelisation}, and financial systems, e.g. \cite{buckwar2025american}.\\
In this paper, we focus on FPT problem for PDifMPs with a time-varying threshold. As PDifMPs provide a flexible and general framework for modelling systems with both continuous evolution and discrete transitions, they are well-suited for a broad range of modern applications. Despite their relevance, FPT methods tailored to PDifMPs remain unexplored, partly because they have been developed recently and their inherent complexity. While time-discretisation schemes can be applied , cf. \cite{buckwar2025numerical}, they tend to overestimate FPTs unless very fine step sizes are used—leading to increased variance, a known issue. Analytical approaches are difficult to implement, as the drift and diffusion coefficients typically change after each jump. Of particular interest is the Exact simulation method introduced in \cite{herrmann2023exact} for jump diffusions, which can be adapted to PDifMPs by accounting for these post-jump changes in the drift and diffusion. 
Extending this framework to time-dependent thresholds presents significant challenges, which we address in this work. In this case, the distribution of the conditioned auxiliary process is no longer available, making it difficult to simulate the value of the continuous part exactly at the jump time when its FPT exceeds or coincides with the next jump. To overcome this, we develop a novel approach to simulate the auxiliary process and derive the corresponding rejection probability.

The paper is organised as follows: In Section~\ref{general setting}, we explain the basics of PDifMPs. Section~\ref{section 3} sets the algorithmic foundations for the computation of FPTs for time-varying thresholds by developing a hybrid scheme for the continuous and the jump parts of a PDifMP. Moreover, we prove well-definedness and convergence of our method. We illustrate our algorithm by examples in Section~\ref{Numa_exp}. Section~\ref{sec:outlook} concludes. 
\section{Generalised PDifMPs}\label{general setting}
\noindent
In this section, we present a generalised formulation of the PDifMP, extending the classical definition introduced in \cite{bujorianu2006toward, meddah2024stochastic, bect2007processus}. This generalised version allows for discontinuities in the PDifMP solution at jump times, thereby broadening the modelling scope of such processes. We provide a detailed description of the structural elements and dynamical mechanisms that characterise the process, which will serve as the foundational framework for the developments in the remainder of the paper.\\
Consider a filtered probability space $(\Omega, \mathcal{F}, (\mathcal{F}_t)_{t \geq 0}, \mathbb{P})$, where the filtration $(\mathcal{F}_t)_{t \geq 0}$ on $\Omega$ satisfying the usual conditions, i.e., $(\mathcal{F}_t)_{t \geq 0}$ is right continuous and $\mathcal{F}_0$ contains all $\mathbb{P}$-null sets.  Let $(B_t)_{t\geq 0} \in \mathbb{R}^m$, $m \in \mathbb{N}$, be a standard Brownian motion defined on this space. The process $B_t$ is assumed to be $\mathcal{F}_t$-adapted.\\
Let $ E \subset \mathbb{R}^n $ denote the state space, defined as $ E = E_1 \times E_2 $, where $ E_1$ and $ E_2 $ are subsets of $\mathbb{R}^{n_1}$ and $\mathbb{R}^{n_2}$, with $n_1+n_2=n$. We equip the state space $E$ with the Borel $\sigma$-algebra $\mathcal{B}(E)$ generated by the product topology of $E_1$ and $E_2$.\\
A generalised PDifMP is a c\`adl\`ag\footnote{These are continuous-time
stochastic processes with sample paths that are almost surely everywhere right continuous with limits from the left existing everywhere.} stochastic process $(U_t)_{t \geq 0}$ taking values in $E$, and written as $U_t = (Y_t, Z_t)$, where $Y_t$ represents the continuous component and $Z_t$ the jump component. We assume that $Z_t$ is piecewise constant and right-continuous, with random jump times $(T_i)_{i \in \mathbb{N}}$, where we set $T_0=0$ by convention. Both components of $(U_t)_{t \geq 0}$ are assumed to be non-exploding. The process $U_t$ is characterised by three local characteristics $(\phi, \lambda, \mathcal{Q})$. The function $\phi$ defines the stochastic flow of the continuous component, while the jump behaviour is described by the rate function $\lambda$ and the transition kernel $\mathcal{Q}$. The dynamics of the PDifMP $(U_t)_{t \geq 0}$ are defined by the characteristic triple as follows.

\subsubsection*{Piecewise continuous dynamics}
\noindent The sequence of jump times $(T_i)_{i \in \mathbb{N}}$ defines a sequence of randomly distributed grid points on $\mathbb{R}_{+}$. Between these jump times, the process $(U_t)_{t \geq 0} $ evolves according to a stochastic flow $ \phi $, defined recursively. For each $ i \in \mathbb{N}_0$\footnote{Here, $ \mathbb{N}_0 = \mathbb{N} \cup \{0\} $ denotes the set of non-negative integers.}, let $ u_i := (y_i, z_i) \in E $ denote the state of the process at time $ T_i $. Then, for all $ t \in [T_i, T_{i+1}) $, the dynamics of $(U_t)_{t \geq 0} $ are given by
\begin{equation*}
\phi(t, u_i) := Y_t^i,
\end{equation*}
where $ Y^i = (Y_t^i)_{t \geq T_i} $ solves the SDE associated with the continuous dynamics on that interval, given by
\begin{equation}
\label{eq: main SDE}
dY^i_t = \mu(t, (Y^i_t, z_i))\, dt + \sigma(t, (Y^i_t, z_i))\, dB_t, \qquad t \in [T_i, T_{i+1}),
\end{equation}
with initial condition at times $T_i$ given by
\begin{equation*}
Y^i_{T_i} = y_i.
\end{equation*}
\begin{notation}
On each inter-jump interval $[T_i, T_{i+1})$, the continuous evolution of the process is denoted by $ Y^i_t $, which solves the SDE \eqref{eq: main SDE} with the discrete state $ z_i $ held fixed, i.e. $Z_t=z_i$ for all $t\in[T_i,T_{i+1})$. The variable $ y_i = Y_{T_i} $ represents the value of the continuous component at the start of the interval. Further, an arbitrary element of the state space $ E $ is denoted by $ u = (y, z)$.
\end{notation}
\begin{assumption}
    The functions $\mu : \mathbb{R}_{+} \times E\to \mathbb{R}^{n_1}$ and $\sigma : \mathbb{R}_{+} \times E\to \mathbb{R}^{n_1 \times m}$ are globally Lipschitz continuous and satisfy linear growth conditions with respect to the state variable $y \in E_1$, uniformly in $t \in \mathbb{R}_{+}$ and $z \in E_2$. Under these conditions, the SDE \eqref{eq: main SDE} has a unique strong solution on each inter-jump interval $ [T_i, T_{i+1}) $.
    \label{existn_&uniq}
\end{assumption}
\subsubsection*{Jump Mechanism}

\noindent
The jump dynamics of the PDifMP $(U_t)_{t\geq 0}$ are governed by the remaining two components of its characteristic triple, namely, the jump rate function $ \lambda$ and the transition kernel $ \mathcal{Q}$.

\begin{assumption}
\label{eq:lambda_integrability}
Let  $ \lambda: E \to \mathbb{R}_+ $ be a measurable function satisfying the following integrability condition
 \begin{equation*}
       \forall \, u\, \in E, \, \, \int_0^\infty \lambda\big(\phi(t, u), z\big)\,dt < \infty.
    \end{equation*}
\end{assumption}
\noindent This ensures that the expected number of jumps remains finite over any finite time horizon.
\noindent Given any initial state $ u_i = (y_i, z_i) \in E $, the jump times $T_i$ are random variable with associated  survival function given by
\begin{equation*}
  \mathcal{S}(t, u_i)= \mathbb{P}_{u_i}(T_{i+1} > t) := \exp\left(-\int_{T_i}^t \lambda\big(\phi(h, u_i), z_i\big)\, dh\right), \qquad t \geq T_i,
\end{equation*}
which expresses the probability that no jump occurs up to time $ t $. If $ T_{i+1} = \infty $, the process evolves continuously according to the stochastic flow
\begin{equation*}
U_t = \big(\phi(t, u_i), z_i\big), \qquad \text{for all } t \geq 0.    
\end{equation*}

\medskip

\noindent If instead $ T_{i+1} < \infty$, the process undergoes a jump. The continuous component is updated by applying a jump size to its pre-jump value:
\begin{equation*}
y_{i+1} := Y^{i}_{T_{i+1}} + j(T_{i+1}, u_i),
\end{equation*}
\noindent where $j : \mathbb{R}_{+} \times E \to \mathbb{R}^{n_1}$ represents the size of the jump. To ensure the well-posedness of the jump updates, we assume the following condition.
\begin{assumption}
\label{ass:jump-size}
The jump size function $ j : \mathbb{R}_+ \times E \to \mathbb{R}^{n_1} $ is measurable and satisfies a linear growth condition in $ y \in E_1 $, uniformly in $ t \in \mathbb{R}_+ $ and $ z \in E_2 $, that is, there exists a positive constant $C$ such that
\begin{equation*}
   \| j\big(t, (y, z)\big) \| \leq C\big(1 + \| y \|\big). 
\end{equation*}
\end{assumption}
\noindent The new value of the discrete component is determined via a Markov kernel $ \mathcal{Q} $, defined as a transition probability from the state space $E$ to $E_2$. More precisely, we sample a uniform random variable $\mathcal{U}_i$ on $[0, 1] $, independent of all other randomness (Brownian motion and jump times), and apply a measurable function ${\psi : [0, 1] \times E \to E_2}$ such that
\begin{equation*}
    \mathbb{P}(\psi(\mathcal{U}_{i+1}, u_i) \in \mathcal{A}) = \mathcal{Q}(u_i, E_1 \times \mathcal{A}), \qquad \text{for all } \mathcal{A} \in \mathcal{B}(E_2).
\end{equation*}
\noindent The discrete component is then updated via
\begin{equation*}
    z_{i+1} := \psi(\mathcal{U}_{i+1}, u_i),
\quad \text{and hence} \quad
U_{T_{i+1}} = (y_{i+1}, z_{i+1}).
\end{equation*}

\begin{assumption}
The transition kernel $ \mathcal{Q} : E \times \mathcal{B}(E_2) \to [0, 1] $ is measurable in its first argument, and for each fixed $ u \in E $, $ \mathcal{Q}(u, \cdot) $ is a probability measure on $ E_2 $.
\end{assumption}

\noindent
We additionally assume that the kernel enforces state changes at each jump, i.e
\begin{equation*}
 \mathcal{Q}\big(u, E-\{u\}\big) = 1 \quad \text{for all } u \in E.
\end{equation*}

\begin{remark}
When $j=0$, the PDifMP $U_t$ evolves without discontinuities and is reset at each jump time by continuation of the flow, i.e., $y_{i+1} = \phi(T_{i+1}, u_i) $. In this case, the dynamics reduce to those of a classical PDifMP \cite{bujorianu2006toward, bect2007processus, meddah2024stochastic}. When $j \neq 0$, jumps induce discontinuities in the path of the process, and the full PDifMP $(U_t)_{t \geq 0} $ is c\'adl\`ag in the classical sense.
\end{remark}
\noindent After the jump, the process restarts from the updated state $ u_{i+1} = (y_{i+1}, z_{i+1}) $, and the procedure repeats recursively.
\noindent We define the embedded Markov chain associated with the PDifMP $(U_t)_{t\geq 0}$ by
\begin{equation*}
 K_i := U_{T_i}, \qquad \omega_i := T_i - T_{i-1}, \quad \text{for } i \in \mathbb{N},   
\end{equation*}
with $T_0:= 0$. The pair $ (K_i, \omega_i) $ records the post-jump locations and the waiting times until the next jump, respectively. The number of jumps up to time $t$ is denoted by
\begin{equation*}
    N_t = \sum_{i \geq 1} \mathbbm{1}_{\{T_i \leq t\}}.
\end{equation*}

\begin{assumption}
\label{jump_count}
For every initial condition $u_i \in E$ and every finite time horizon $t > 0$, the expected number of jumps remains finite, i.e.,
\begin{equation*}
  \mathbb{E}[N_t \mid U_t = u_i] < \infty.  
\end{equation*}
\end{assumption}

\noindent
This non-explosion condition guarantees global existence of sample paths. Throughout this work, we assume that the PDifMP $ (U_t)_{t \geq 0} $ satisfies Assumptions~\ref{existn_&uniq}--\ref{jump_count}, ensuring that the model is well-posed.

\section{Exact Simulation of the FPT of PDifMPs}
\label{section 3}
In this section, we define the FPT of PDifMPs to time-dependent thresholds, extending the framework introduced in \cite{herrmann2023exact} from constant to general thresholds and PDifMP formulation. Consider the generalised PDifMP $ U_t=(Y_t, Z_t)$ defined in Section \ref{general setting}, where $Y_t$ evolves continuously between jump times and $Z_t$ is the discrete component determined by the transitional kernel $\mathcal{Q}$.\\
Throughout this work, we restrict to the one-dimensional setting, i.e., $m = n_1 = n_2 = 1$, and consider first passage time of the continuous component $Y_t$ of the PDifMP. We assume the process starts below the threshold, that is, $Y_0 = \phi(0, U_0) < \tilde{\beta}(0)$
 where $\tilde{\beta}$ is a deterministic, time-varying threshold.
\begin{remark}
Although the PDifMP is a two-component process, the FPT we consider is determined solely by the continuous component. This is because the threshold is defined over the continuous state space, and the constant value of the process between jumps does not influence whether the threshold has been crossed. Furthermore, the flow map $\phi(t, U_t) = Y_t$ fully describes the evolution of the continuous component between jumps. Therefore, computing the first time $Y_t \geq \tilde{\beta}(t)$ is equivalent to finding the first time the full process $(U_t)_{t \geq 0}$ enters the set $\{ (y, z) \in E : y \geq \tilde{\beta}(t) \}$.
\end{remark}
\noindent The FPT $\tau_{\tilde{\beta}}$ of PDifMP $(U_t)_{t \geq 0}$ is then defined as
\begin{equation}
    \tau_{\tilde{\beta}} = \inf \{ t \in \mathbb{R}^{+}_{0} \mid Y_t \geq \tilde{\beta}(t) \},
    \label{eq:FPT}
\end{equation}
denoting the earliest time at which the continuous component $Y_t$  reaches or crosses the threshold. Note that the stopping time $\tau_{\tilde{\beta}}$ is a random variable, as its realisation depends on the various possible paths of the process $Y_t$.

\subsection{Hybrid exact FPT simulation scheme for PDifMPs}
\label{hybrid_exact_algorithm}

\noindent Our goal here is to develop an Exact simulation strategy for the FPT of a PDifMP to a time-dependent threshold. The idea is to combine existing Exact methods for simulating the FPT of diffusions \cite{khurana2024exact,herrmann2019exact} with a jump mechanism that governs the PDifMP dynamics. To do so, we introduce a simulation scheme, inspired by \cite{herrmann2019exact}, that alternates between continuous diffusion phases and discrete jump updates. We refer to this as the \textit{Hybrid exact FPT simulation scheme}.
\noindent The method proceeds recursively across inter-jump intervals. For each interval $[T_i, T_{i+1})$, we simulate the FPT of a diffusion process with fixed parameters, then compare it to the next jump time. If the threshold is reached first, we stop. Otherwise, we simulate the post-jump value and iterate.

\begin{enumerate}[label=Step \arabic*., leftmargin=*]
    \item \textbf{Simulate the continuous FPT:}  
    Let $z_i$ be the value of the discrete component at the current jump time $T_i$. Define the \textit{tracking process} $Y^{i,\infty}$ as the solution to the SDE
    \begin{equation}
    dY^{i,\infty}_t = \mu\big(t,\, (Y^{i,\infty}_t, z_i)\big)\, dt + \sigma\big(t,\,(Y^{i,\infty}_t, z_i)\big)\, dB_t, \qquad t \in [T_i, \infty),
    \label{eq:tracking process}
    \end{equation}
    with $Y^{i,\infty}_{T_i} = y_i$. Simulate the FPT
    \begin{equation*}
        \tau_{\beta}^{\text{cont}, i} := \inf\left\{ t \geq T_i \, \middle| \, Y^{i,\infty}_t = \tilde{\beta}(t) \right\},
    \end{equation*}
    using the Exact method from \cite{khurana2024exact}. Details are given in Subsection~\ref{subsection3.2}.

    \item \textbf{Simulate the next jump time:}  
    Generate the next jump time $T_{i+1} = T_i + \tau_i$ based on the inter-jump time $\tau_i$, using the intensity function $\lambda$ as described in Subsection~\ref{general setting}.

    \item \textbf{Compare stopping criteria:}
    \begin{enumerate}[label=(\alph*)]
        \item If $ \tau_{\tilde{\beta}}^{\text{cont}, i} < T_{i+1}$, the process reaches the threshold before the jump. Set
        \begin{equation*}
            \tau_{\tilde{\beta}} = \tau_{\tilde{\beta}}^{\text{cont}, i},
        \end{equation*}
        and stop.

        \item If instead $ \tau_{\tilde{\beta}}^{\text{cont}, i} \geq T_{i+1} $, continue with the following:
        \begin{enumerate}[label=(\roman*)]
            \item Simulate the value of $Y^{i,\infty}$ at $T_{i+1}$ using the conditional Exact sampling method. This step involves technical considerations due to the time-dependent threshold and is detailed in Subsection~\ref{subsection3.3}.

            \item Apply the jump update to the continuous component:
            \begin{equation}
            Y_{T_{i+1}} = Y^{i,\infty}_{T_{i+1}} + j\big(T_{i+1}, \, (Y^{i,\infty}_{T_{i+1}}, z_i)\big).
            \label{eq:value_after_jump}
            \end{equation}

            \item If $ Y_{T_{i+1}} \geq \tilde{\beta}(T_{i+1}) $, then the threshold is crossed due to the jump. Set
            \begin{equation*}
                \tau_{\tilde{\beta}} = T_{i+1},
            \end{equation*}
           
            and stop.

            \item Otherwise, update the state to $u_{i+1}$ and repeat from Step 1.
        \end{enumerate}
    \end{enumerate}
\end{enumerate}

\noindent
The algorithm continues until the threshold is reached. Formally, this  corresponds to finding the first index $i\in \mathbb{N}_{0}$ such that
\begin{equation}
\tau_{\tilde{\beta}} = \min\left\{ \tau_{\tilde{\beta}}^{\text{cont}, i} \,\middle|\, \tau_{\tilde{\beta}}^{\text{cont}, i} < T_{i+1} \right\}.
\end{equation}

\begin{remark}
In this context, the tracking process $Y^{i,\infty}$ is introduced to decouple the continuous diffusion dynamics from the discrete jump mechanism. This auxiliary process evolves as a standard SDE with frozen discrete state $z_i$, and serves as a tool to compute the FPT. Unlike the actual processes $Y_t$ and the discretised version $Y^{i}_t$, which may be interrupted by jumps, $Y^{i,\infty}$ provides an uninterrupted diffusion trajectory that allows us to determine whether the threshold is reached before the next jump. More importantly, it also satisfies the regularity requirements needed to apply the Girsanov transformation, which is the core technique behind the Exact simulation algorithm.
\end{remark}

\noindent In Figure~\ref{fig:FPT_scenarios} we illustrate the simulation procedure introduced above. The left panel (\textbf{Scenario 1}) corresponds to the case where the continuous dynamics alone drive the process across the threshold before any jump occurs. In this case, the FPT coincides with the first-passage time $\tau_{\tilde{\beta}}^{\text{cont}, i}$ of the diffusion process $Y^{i,\infty}$ and can be sampled directly using the Exact method from \cite{khurana2024exact}.

\begin{figure}[H]
    \centering
    \begin{subfigure}[b]{0.495\textwidth}
        \centering
        \includegraphics[width=\textwidth]{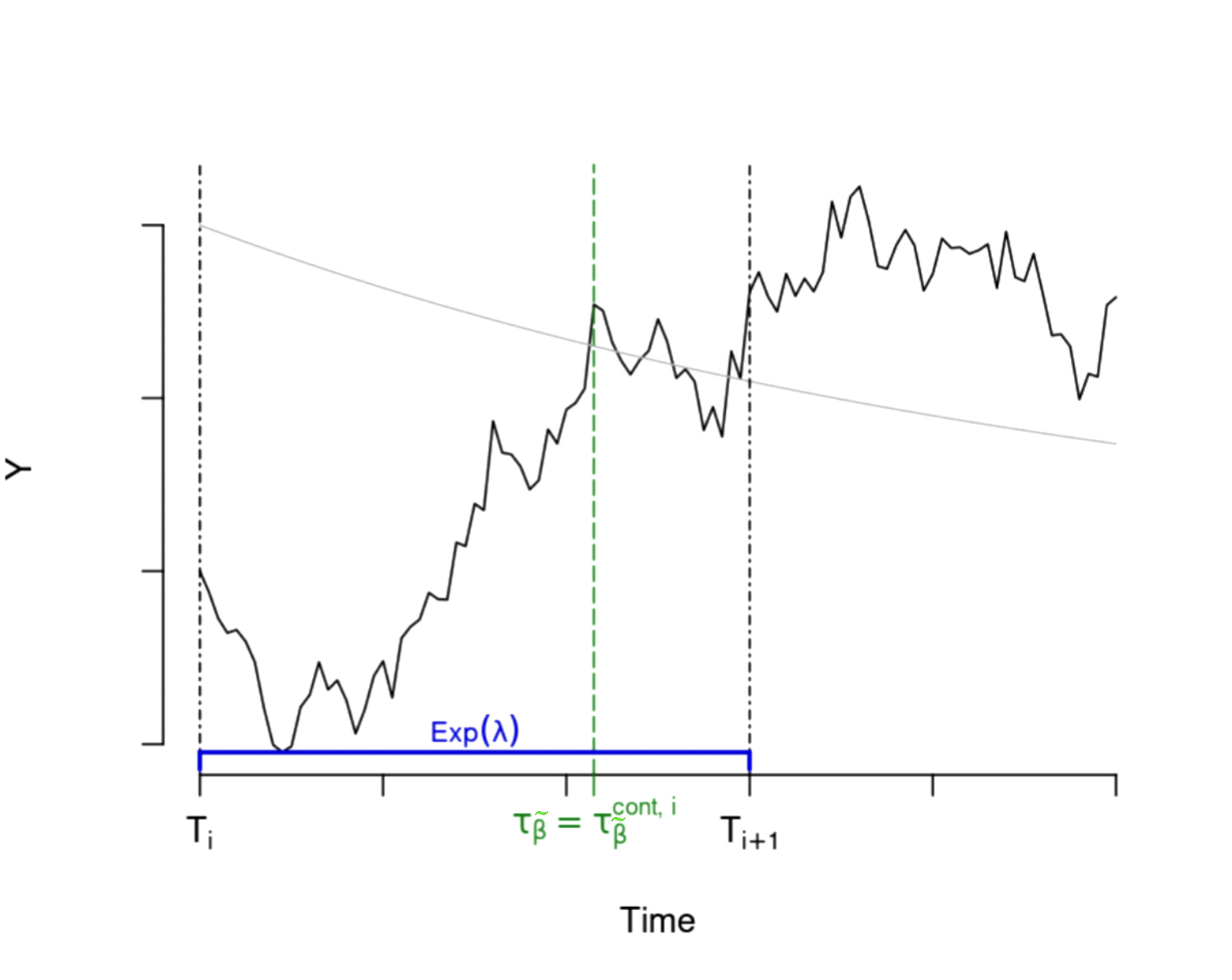}
        \caption{\textbf{Scenario 1}. The process crosses the threshold before the next jump.}
    \end{subfigure}
    \hfill
    \begin{subfigure}[b]{0.495\textwidth}
        \includegraphics[width=0.82\textwidth]{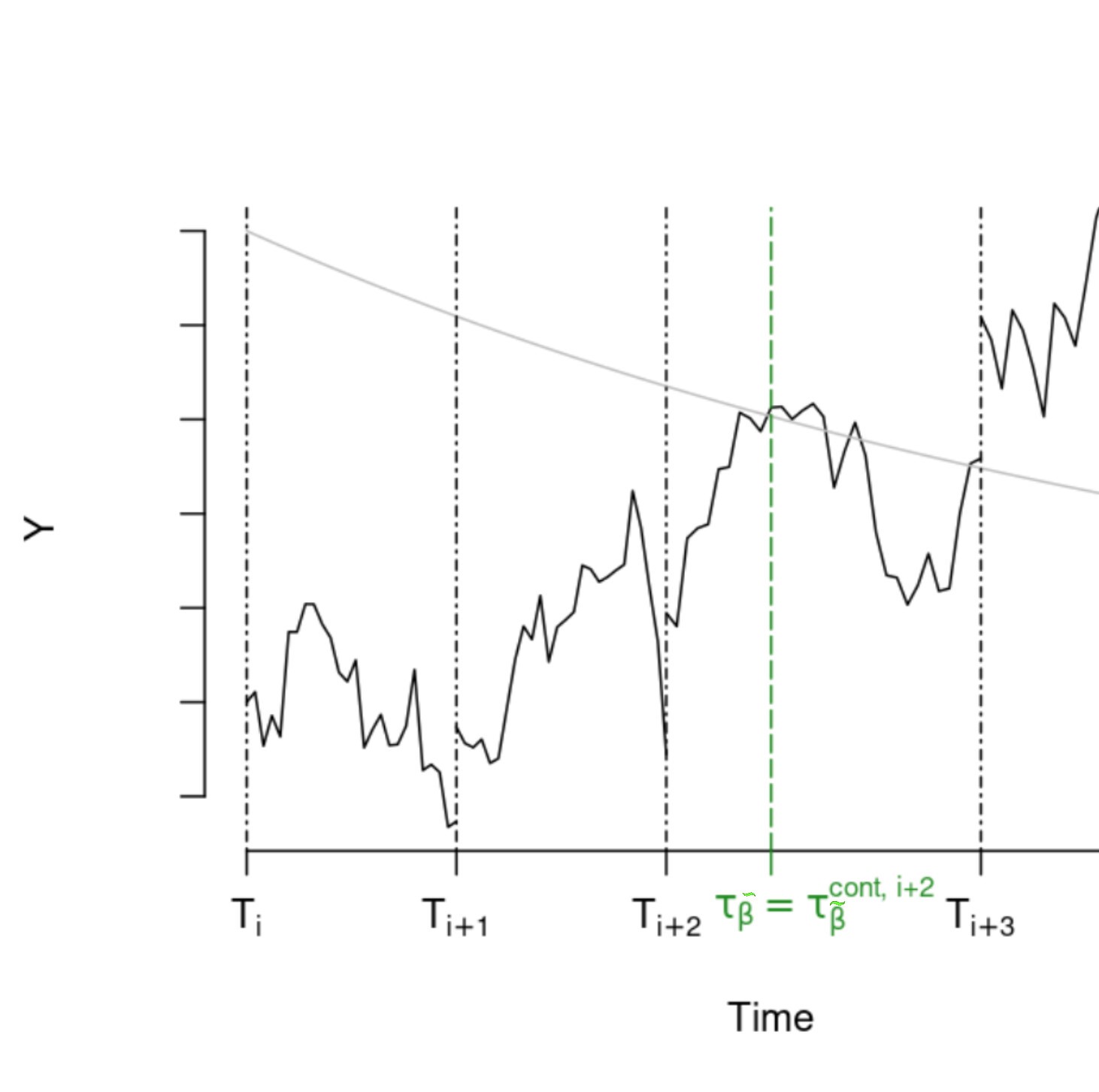}
        \caption{\textbf{Scenario 2}. The process jumps before reaching the threshold.}
    \end{subfigure}
     \caption{Illustration of the two possible first-passage time scenarios.}
    \label{fig:FPT_scenarios}
\end{figure}

\noindent In contrast, the right panel (\textbf{Scenario 2}) depicts a situation in which a jump occurs before the threshold is reached. The process value at the jump time $T_{i+1}$ is sampled (see Subsection~\ref{subsection3.3} for the detailed approach) and updated via the jump rule \eqref{eq:value_after_jump}. The resulting state becomes the new starting point for the next inter-jump interval. This recursive procedure is repeated until the threshold is crossed. In the example shown, the threshold is crossed during the interval $[T_{i+2}, T_{i+3})$, hence $\tau_{\tilde{\beta}} = \tau_{\tilde{\beta}}^{\text{cont}, i+2}$.\\
\noindent This recursive structure highlights the two core components of the simulation algorithm:
\begin{enumerate}[label=(\roman*)]  
    \item Exact simulation of the first-passage time $\tau_{\tilde{\beta}}^{\text{cont},i}$ of the tracking process \eqref{eq:tracking process}, extended in the upcoming section to time- and state-dependent coefficients.
    \item Exact sampling of $Y^{i,\infty}_{T_{i+1}}$ conditioned on $\tau_{\tilde{\beta}}^{\text{cont},i} \geq T_{i+1}$, developed in Subsection~\ref{subsection3.3}.
\end{enumerate}
\noindent
By repeating this procedure and recording the outcomes of $\tau_{\tilde{\beta}}$, one obtains an empirical approximation of the FPT distribution for the full PDifMP $(U_t)_{t\geq 0}$.
\subsection{FPT Simulation of the tracking process}
\label{subsection3.2}

We now detail the first key step of the algorithm introduced in Section~\ref{hybrid_exact_algorithm} namely, the simulation of the FPT $\tau_{\tilde{\beta}}^{\text{cont},i}$ of the tracking process $Y^{i,\infty}$ defined in Equation \eqref{eq:tracking process}.\\
\noindent Our approach builds on the Exact method developed in \cite{khurana2024exact} for SDEs, which we extend here to cover general time- and state-dependent drift and diffusion coefficients. The method is based on rejection sampling using Brownian motion as an auxiliary process, where acceptance is governed by a version of Girsanov’s theorem adapted to stopping times.\\
\noindent The main technical contribution of this section is to show that this Girsanov-based rejection scheme remains valid under our generalised setting. Moreover, we provide here an explicit expression for the rejection probability. As in earlier works, see for instance \cite{jeanblanc2009mathematical}, the application of Girsanov’s transformation requires the diffusion coefficient to be unitary. We therefore first transform the SDE via the Lamperti transform. To apply the Lamperti transform, we require the following regularity condition on the diffusion coefficient. 
\begin{assumption}
        The diffusion coefficient $\sigma(t, (y, z))$ is non-negative for all $t \in [T_i, \infty)$, $y \in E_1$, and fixed $z \in E_2$.
  \label{assuption_for_Lamperti}
\end{assumption}
\noindent Combined with Assumption~\ref{existn_&uniq}, this ensures that the transformation is well-defined and invertible.

\begin{proposition}[Lamperti transformation]
\label{Lampt}
Suppose Assumption~\ref{assuption_for_Lamperti} holds. Fix a discrete state $z_i \in E_2$ and define the time-dependent bijective map
\begin{equation*}
  F^{i}\big(t,\, (Y^{i,\infty}_{t}, z_i)\big) := \int^{y} \frac{1}{\sigma\big(t, \, (h, z_i)\big)} dh \bigg|_{y=Y^{i,\infty}_t},
  \end{equation*}
which transforms $Y^{i,\infty}_t$ to the process $X^{i}_t := F\big(t,\, (Y^{i,\infty}_t, z_i)\big)$ satisfying
\begin{equation}
dX^{i,\infty}_t = \alpha\big(t,\, (X^{i,\infty}_t,z_i)\big)dt + dB_t, \quad   X^{i,\infty}_{T_{i}} = F^i\big(T_i,\,( Y^{i,\infty}_{T_{i}},z_i)\big).
\label{eq:cont_SDE_after_Lamperti}
\end{equation}
Here, the transformed drift $\alpha$ is given by
\begin{equation*}
    \alpha\big(t,\, (x,z_i)\big) = \left.\left(\frac{\partial F^i}{\partial t}\big(t,\, (y,z_i)\big)+ \frac{\mu\big(t,\, (y,z_i)\big)}{\sigma\big(t,\, (y,z_i)\big)} - \frac{1}{2}\frac{\partial \sigma}{\partial y}\big(t,\, (y,z_i)\big) \right)\right|_{y={F^i}^{-1}\big(t,\, (x,z_i)\big)}.
\end{equation*}
\end{proposition}
\noindent The transformed process $X^{i,\infty}_t$ inherits the same probability space $(\Omega, \mathcal{F}, (\mathcal{F}_t)_{t \geq 0}, \mathbb{P})$ as the original process $Y^{i,\infty}_t$.
\begin{remark}
Since the Lamperti transform is bijective in the continuous variable $y$, computing the FPT of $Y^{i,\infty}$ to the threshold $\tilde{\beta}(t)$ is equivalent to computing the FPT of $X^{i,\infty}$ to the transformed threshold
    \begin{equation*}
        \beta(t) := F^i\big(t,\, (\tilde{\beta}(t), z_i)\big), \quad \text{for fixed } z_i\in E_2.
    \end{equation*}
\end{remark}
\noindent We define the notation $\tau_{\beta}^{\text{cont},i}$ as the FPT of the process $X^{i,\infty}$ to $\beta$.
 \noindent To proceed, we now present a stopping-time version of Girsanov’s theorem that is well-defined for SDEs with unit diffusion and time- and space-dependent drift, as in Equation \eqref{eq:cont_SDE_after_Lamperti}.
\begin{theorem}[Girsanov's transformation for stopping times]
\label{thm:girsanov_stopping}
Let $(\Omega, \mathcal{F}, (\mathcal{F}_t)_{t \geq 0}, \mathbb{P})$ be a filtered probability space carrying a standard Brownian motion $(B_t)_{t \geq 0}$. Let $(X^{i,\infty}_t)_{t \geq 0}$ be the unique strong solution of the SDE \eqref{eq:cont_SDE_after_Lamperti}, where the drift function $\alpha: \mathbb{R}_{+} \times \mathbb{R} \to \mathbb{R}$ is $(\mathcal{F}_t)$-adapted.\\
\noindent Let $\rho$ be the explosion time of $X^{i,\infty}_t$, and let $\mathcal{T}$ be any stopping time such that $\mathcal{T} < \rho$ almost surely. If $X^{i,\infty}_t$ remains bounded, then $\rho=\infty$. Assume the Novikov condition holds up to time $\mathcal{T}$, i.e.
\begin{equation}
\mathbb{E}_{\mathbb{P}}\left[\exp\left( \frac{1}{2} \int_0^{\mathcal{T}} \alpha^2\big(t,\,( X^{i,\infty}_t, z_i)\big)\, dt \right) \right] < \infty.
\label{novik}
\end{equation}
Define the process
\begin{equation}
    M_t = \exp \left( -\int_0^t \alpha\big(h, (X^{i,\infty}_h,z_i)\big) dB_h - \frac{1}{2} \int_0^t \alpha^2\big(h,\,( X^{i,\infty}_h,z_i)\big) dh \right).
    \label{Exp-process}
\end{equation}
Then, the probability measure $\mathbb{Q}$ given by
\begin{equation}
\left.\frac{d\mathbb{Q}}{d\mathbb{P}}\right|_{\mathcal{F}_{\mathcal{T}}} = M_{\mathcal{T}}.
\label{measure_transf}
\end{equation}
is well-defined and equivalent to $\mathbb{P}$ on $\mathcal{F}_{\mathcal{T}}$. Under $\mathbb{Q}$, the process
\begin{equation}
\tilde{B}_t:= B_t + \int_0^t \alpha\big(h, \,(X^{i,\infty}_h,z_i)\big)\, dh, \quad 0 \leq t \leq \mathcal{T},
\label{trans_proc}
\end{equation}
is a standard Brownian motion. In particular, the process satisfies
\begin{equation*}
 dX^{i,\infty}_t = d\tilde{B}_t.   
\end{equation*}
Moreover, for any bounded, measurable functional $\Psi: C([0, \mathcal{T}]; \mathbb{R}) \to \mathbb{R}$, it holds that
\begin{equation}
\mathbb{E}_{\mathbb{P}}\left[\Psi\left( X^{i,\infty}_t \, ; \, 0 \leq t \leq \mathcal{T} \right)\right] 
= \mathbb{E}_{\mathbb{Q}}\left[\Psi\left( X^{i,\infty}_t \, ; \, 0 \leq t \leq \mathcal{T} \right) M_{\mathcal{T}} \right].
\label{meas-transf}
\end{equation}
\end{theorem}
\noindent This theorem extends classical results (e.g., \cite{oksendal2013stochastic, jeanblanc2009mathematical}) by allowing time- and space-dependent drift and stopping times. Under the transformed measure $\mathbb{Q}$, the process $X^{i,\infty}$ behaves as a standard Brownian motion, thereby establishing a probabilistic equivalence between $X^{i,\infty}$ and Brownian motion. This result is key to the Exact simulation framework, as it enables computing the acceptance probability for a Brownian path to be a valid sample of $X^{i,\infty}$ using the weight $M_{\mathcal{T}}$. The proof of Theorem~\ref{thm:girsanov_stopping} is outlined in Appendix~\ref{Appendix_A}.
\begin{remark}
    To ensure that the explosion time $\rho$ occurs after the stopping time $\mathcal{T}$, one must verify this condition using Feller's Test for Explosions; see \cite{karatzas1991brownian}.
    \label{rmk : Feller test}
\end{remark}
To apply Girsanov’s theorem in our setting, we require certain regularity conditions. Let ${C^{1,1}([T_i, \infty) \times E)}$ denote the set of functions $f : [T_i, \infty) \times E \to \mathbb{R}$ such that, for each fixed $z \in E$, the partial derivatives $\partial_t f(t, (x, z))$ and $\partial_x f(t, (x, z))$ exist and are continuous on $[T_i, \infty) \times E$.
\begin{assumption}
\label{ass: on drift}    
Let $\alpha : [T_i,\infty) \times E \to \mathbb{R}$ be the drift function. We assume:
\begin{enumerate}[label=(\roman*)]
    \item For each fixed $z_i \in E_2$, the function $\alpha(t, (x, z)) \in C^{1,1}([T_i, \infty) \times E)$;
    \item The following integrability condition holds
    \begin{equation*}
        \mathbb{E}_{\mathbb{Q}}\left[\int_{T_i}^{\mathcal{T}} \alpha^2\big(t, (X^{i,\infty}_t, z_i)\big)\, dt\right] < \infty;
    \end{equation*}
        \item The threshold function $\beta : [T_i, \infty) \to \mathbb{R}$ is differentiable.
\end{enumerate}
\end{assumption}
\noindent We also introduce the following functions
\begin{align*}
A\big(t, (x, z_i)\big) &:= \int^x \alpha\big(t, \, (h, z_i)\big)\, dh, \\[0.22cm]
\gamma_1(t, z_i) &:= -\frac{\partial}{\partial t} A\big(t,\, (\beta(t), z_i)\big) - \alpha\big(t,\, (\beta(t), z_i)\big)\, \beta'(t), \\[0.22cm]
\gamma_2\big(t,\,( x, z_i)\big) &:= \frac{\partial}{\partial t} A\big(t,\, (x, z_i)\big) + \frac{1}{2} \left( \frac{\partial \alpha}{\partial x}\big(t,\,( x, z_i)\big) + \alpha^2\big(t,\,( x, z_i)\big) \right).
\end{align*}
\noindent Here, $A$ denotes the antiderivative of the drift $\alpha$ with respect to $x$ for fixed $z_i\in E_2$.\\
For notational convenience, we denote by $\tilde{X}^i_t := \left(X^{i,\infty}_t\right)_{\mathbb{Q}}$ the process obtained by applying the Girsanov transformation from Theorem~\ref{thm:girsanov_stopping} to $X^{i,\infty}_t$. That is, under the measure $\mathbb{Q}$, the process $\tilde{X}^i$ evolves as a standard Brownian motion starting from $X^{i,\infty}_{T_i}$.
\begin{theorem}
\label{thm:Girsanovprob_FPT}
Suppose that the Novikov condition \eqref{novik} and Assumption~\ref{ass: on drift} are satisfied. Then, for any bounded measurable function $\Psi: \mathbb{R} \to \mathbb{R}$, the following identity holds
\begin{equation}
\mathbb{E}_{\mathbb{P}}\left[\Psi(\tau_{\beta}^{\text{cont},i}) \, \mathbf{1}_{\{\mathcal{T}_\beta < \infty\}}\right] 
= \exp\left(-A\big(T_i,\,( \tilde{X}^i_{T_i}, z_i)\big) + A\big(T_i,\, (\beta(T_i), z_i)\big)\right)
\cdot \mathbb{E}_{\mathbb{Q}}\left[\Psi(\tau_{\beta}^{\text{cont},i}) \cdot \eta^{[1]}(\mathcal{T}_\beta)\right],
\label{eq:Girsanov2}
\end{equation}
where the probability weight $\eta^{[1]} : \mathbb{R}_+ \to \mathbb{R}$ is defined as
\begin{equation}
\eta^{[1]}(t) := \mathbb{E}_{\mathbb{Q}}\left[\exp\left(-\int_{T_i}^{t} \left( \gamma_1(h, z_i) + \gamma_2\big(h, (\tilde{X}^i_h, z_i)\big) \right) dh \right) \,\Big|\, \ t=\tau_\beta^{\text{cont},i} \right].
\label{eq:eta_def}
\end{equation}
\end{theorem}
\noindent The proof of this theorem can be found in Appendix~\ref{Appendix_B}.

\begin{remark}
\begin{enumerate}[label=(\roman*)]
    \item Equation~\eqref{eq:Girsanov2} decomposes the transformed expectation into three parts: $\gamma_1$ captures the effect of the moving boundary $\beta(t)$; $\gamma_2$ reflects the influence of the nonlinear drift $\alpha\big(t, (x, z_i)\big)$; and the exponential prefactor corrects for the mismatch between the initial value $\tilde{X}^{i}_{T_i}$ and the threshold $\beta(T_i)$. This decomposition is exact under the stated assumptions, with all terms being $\mathcal{F}_{\mathcal{T}_\beta}$-measurable.
    \item In Theorem \ref{thm:Girsanovprob_FPT}, Girsanov’s Theorem is used to transform the dynamics from $\mathbb{P}$ to $\mathbb{Q}$. In the context of Exact method we apply it in reverse, we start with simulating the FPT of Brownian motion and evaluate the probability $\eta^{[1]}(t)$ to perform acceptance rejection. 
\end{enumerate}
\end{remark}
\noindent To interpret $\eta^{[1]}(t)$ as a well-defined probability, we require the following assumptions.
\begin{assumption}
\label{ass:reweighting}
The terms $\gamma_1$ and $\gamma_2$ satisfy:
\begin{enumerate}[label=(\roman*)]
    \item $\gamma_1$ and $\gamma_2$ are non-negative functions;
    \item There exists a constant $\kappa > 0$ such that
    \begin{equation*}
        \sup_{t,x} \left\{ \gamma_1(t, z_i) + \gamma_2\big(t,\,( x, z_i)\big) \right\} \leq \kappa.
    \end{equation*}
\end{enumerate}
\end{assumption}
\noindent The direct evaluation of $\eta^{[1]}(t)$ is not straightforward, as it depends on the entire continuous path of the transformed process $\tilde{X}^i_t$. To address this, we apply a Poisson thinning technique that only requires evaluation at finitely many time points.
\begin{theorem}[Poisson thinning representation]
\label{thm:poisson_thinning}
Let $\tau^{*}$ denote the FPT of a standard Brownian motion starting at time $T_i$ with initial value $X^{i,\infty}_{T_i}$. Let $\Phi$ be a homogeneous Poisson point process of unit intensity on the rectangle
\begin{equation*}
   \mathbf{D} := [T_i, \tau^{*}) \times [0, \kappa] \subset \mathbb{R}^2, 
\end{equation*}
where $\kappa > 0$ is the uniform upper bound from Assumption~\ref{ass:reweighting}. Given a trajectory $(\tilde{X}^i_t)_{t \in [T_i, \tau^{*})}$, define the subgraph
\begin{equation*}
 \mathcal{G} := \left\{ (t, h) \in \mathbf{D} \,\middle|\, 0 \leq h \leq \gamma_1(t, z_i) + \gamma_2\big(t,\, (\tilde{X}^i_t, z_i)\big) \right\}.   
\end{equation*}
Let $N := \Phi(\mathcal{G})$ denote the number of Poisson points falling under the graph of $\gamma_1 + \gamma_2$. Then, conditional on the path $(\tilde{X}^i_t)_{t \in [T_i, \tau^{*})}$, we have
\begin{equation*}
 \mathbb{P}\left(N = 0 \,\middle|\, \{\tilde{X}^i_t\}_{t \in [T_i, \tau^{*})} \right)
= \exp\left( -\int_{T_i}^{\tau^{*}} \left[ \gamma_1(t, z_i) + \gamma_2\big(t,\, (\tilde{X}^i_t, z_i)\big) \right] dt \right).   
\end{equation*}
\end{theorem}

\noindent To simulate a reweighted FPT sample, we proceed in three steps. First, a candidate time $\tau^{*}$ is drawn, either from its exact distribution (when available) or using an algorithmic scheme such as in \cite{herrmann2016first}. Next, a Brownian path $(\tilde{X}^i_t)_{t \in [T_i, \tau^{*}]}$ is constructed to satisfy $\tilde{X}^i_{\tau^{*}} = \beta(\tau^{*})$ while remaining below the boundary beforehand. This is achieved by simulating a Bessel bridge:
\begin{equation}
    R_t := \beta(\tau^{*} - t) - \tilde{X}^i_{\tau^{*}}, \qquad T_i \leq t \leq \tau^{*},
    \label{eq:BB_def}
\end{equation}
which is generated exactly and then inverted to obtain the corresponding $\tilde{X}^i_t$ path. Note that the Bessel bridge remains strictly positive by construction \cite{hernandez2013hitting}. Finally, the sample is reweighted using the Poisson thinning to approximate the exponential term in Equation~\eqref{eq:eta_def}, and accepted accordingly. More details are available in \cite{khurana2024exact, herrmann2019exact}.
\begin{remark}
The Bessel bridge construction ensures that the process $\tilde{X}^i_t$ remains below the moving boundary $\beta(t)$ up to $\tau^{*}$, with exact hitting at the endpoint. This guarantees that the first-passage condition is satisfied by construction.
\end{remark}
\begin{algorithm}[H]
\caption{Exact simulation of the FPT for the continuous dynamics}
\label{alg:fpt_simulation}
\begin{algorithmic}[1]
\Require $\beta(T_i) > X^{i,\infty}_{T_i}$
\State \textbf{Input:} candidate FPT $\tau^*$ of Brownian motion to $\beta(t)$; Poisson bound $\kappa$
\vspace{0.1cm}
\State initialise: $l \gets (0,0,0)$; $\mathcal{W} \gets 0$; $\mathcal{E}_0 \gets 0$; $\mathcal{E}_1 \sim \text{Exp}(\kappa)$
\vspace{0.1cm}
\While{$\mathcal{E}_1 \leq (\tau^* - T_i)$ and $\mathcal{W} = 0$}
    \State sample: $G \sim \mathcal{N}(0,I_3)$; $e \sim \text{Exp}(\kappa)$; $\mathcal{U} \sim \mathcal{U}(0,1)$
    \vspace{0.1cm}
    \State update:
    \begin{equation*}
       l \gets \frac{\tau^* - T_i - \mathcal{E}_1}{\tau^* - T_i - \mathcal{E}_0} \cdot l + \sqrt{\frac{(\tau^* - T_i - \mathcal{E}_1)(\mathcal{E}_1 - \mathcal{E}_0)}{\tau^* - T_i - \mathcal{E}_0}} \cdot G 
    \end{equation*}
    and
    \begin{equation*}
      \xi \gets \beta(\mathcal{E}_1 + T_i) - \left\Vert \frac{\mathcal{E}_1(\beta(T_i) - X^{i,\infty}_{T_i})}{\tau^* - T_i}(1,0,0) + l \right\Vert  
    \end{equation*}
    \If{$\kappa \cdot \mathcal{U} \leq \gamma_1\big((\mathcal{E}_1 + T_i),z_i\big) + \gamma_2\big((\mathcal{E}_1 + T_i),\,(\xi,z_i)\big)$}
        \State $\mathcal{W} \gets 1$ \Comment{reject sample}
    \Else
        \State $\mathcal{E}_0 \gets \mathcal{E}_1$ and $\mathcal{E}_1 \gets \mathcal{E}_1 + e$
    \EndIf
\EndWhile

\If{$\mathcal{W} = 0$}
    \State set $\tau_{\text{out}} \gets \tau^*$ \Comment{accept sample}
\Else
    \State repeat with new $\tau^*$ \Comment{rejection sampling}
\EndIf
\State \textbf{Output:} $\tau_{\text{out}}$
\end{algorithmic}
\end{algorithm}
\begin{remark}
\begin{enumerate}
    \item The algorithm relies on three key components; $(i)$ a 3-dimensional Bessel bridge construction to ensure that $\tilde{X}_t^i < \beta(t)$ for all $t < \tau^{*}$, with exact hitting at the endpoint $\tilde{X}_{\tau^{*}}^i = \beta(\tau^{*})$; $(ii)$ Poisson thinning with rate $\kappa$ to evaluate the exponential weight $\eta^{[1]}(t)$ and $(iii)$ an exact rejection sampling scheme to preserve the law of the FPT.  For a detailed theoretical justification, we refer the interested reader to \cite{herrmann2019exact, beskos2006retrospective, beskos2005exact}.
     \item The output $\tau_{\text{out}}$ is an exact sample from the desired FPT distribution $\tau_\beta^{\text{cont}, i}$, as established in \cite{khurana2024exact}.
\end{enumerate}
\end{remark}

\subsection{Simulation of the tracking process at the jump}
\label{subsection3.3}

\noindent This section addresses Step~3.b.i of the hybrid exact FPT simulation scheme introduced in Subsection~\ref{hybrid_exact_algorithm}. The goal is to simulate the value of the tracking process $X^{i,\infty}$ at the next jump time $T_{i+1}$, conditional on the event that the threshold is not crossed before this time, i.e.
\begin{equation*}
    \left\{ X^{i,\infty}_{T_{i+1}} \,\middle|\, \tau_{\beta}^{\text{cont},i} \geq T_{i+1} \right\}.
\end{equation*}
\noindent While we have already established an Exact simulation method for the FPT of $X^{i,\infty}$, the conditional sampling procedure above introduces additional complexity.\\
\noindent To address this, we propose a rejection-based technique that enables exact sampling under the constraint $\tau_{\beta}^{\text{cont},i} \geq T_{i+1} $. The method involves two main components:
\begin{itemize}
    \item[1.] \textit{an auxiliary process:} a Brownian path conditioned to stay below the threshold up to the next jump time;
    \item[2.] \textit{factorised Girsanov weight:} a product of exponential terms derived from the Girsanov transformation, split across time intervals and evaluation points (e.g., jump), and used to determine path acceptance.
\end{itemize}
\noindent These components are challenging to construct in the presence of a time-dependent threshold. For the particular case of a constant threshold, the distribution of the conditioned Brownian path is explicit, as discussed in \cite{herrmann2023exact}. 
In the following, we present the procedure for simulating the conditional Brownian path
\begin{equation}
\left\{ (W_t)_{T_{i} \leq t \leq T_{i+1}} \,\middle|\, \tau^{*} \geq T_{i+1} \right\}.
\label{eq:aux_process}
\end{equation}

\begin{remark}
Although our final goal is to obtain the value of the process $Y^{i,\infty}_{T_{i+1}}$, we first simulate the corresponding value $X^{i,\infty}_{T_{i+1}}$ of the Lamperti-transformed process. The original value can then be recovered using the inverse transformation from Proposition~\ref{Lampt}.
\end{remark}
\noindent The key idea is to generate a point from a sample Brownian path that remains below the threshold $\beta(t)$ until time $T_{i+1}$. This is a nontrivial conditioning problem due to the time-dependence of the threshold. However, we use the fact that the FPT of Brownian motion to a straight line is analytically tractable, see \cite{karatzas1991brownian} for more details. We construct linear thresholds initialising from the threshold $\beta$ with a chosen slope that is strictly less than the minimum slope of $\beta$, ensuring that the Brownian motion stays below the original threshold $\beta$. This construction helps tracking the path of Brownian motion and consequently giving a point from a path of such conditional process. Using this point we later generate the entire path. More precisely, the procedure is structured as follows.
\begin{enumerate}
    \item Choose a constant slope $s \in \mathbb{R}$ such that
    \begin{equation*}
        s < \inf_{t \geq T_i} \beta^{'}(t).
    \end{equation*}
    This ensures that all constructed linear barriers lie strictly below the threshold $\beta(t)$.
   \item  Starting from $(T_i, \beta(T_i))$, define the line $b_1(t)$ with slope $s$. Simulate the FPT $t_1$ of a Brownian motion to this line.
   \item Compare with $T_{i+1}$:
    \begin{itemize}
        \item[i] If $t_1 \geq T_{i+1}$, reduce the slope $s$ further and repeat Step 2. This step may need to be repeated multiple times, decreasing $s$ until $t_1< T_{i+1}$. Moreover, if $T_{i+1}-T_i$ is very small, this adjustment may fail even for small $s$. In this case, a minimum slope $s_{\min}$ is enforced to avoid infinite reduction of $s$.
        \item[ii] If $t_1 < T_{i+1}$, proceed to the next step.
    \end{itemize}
    \item Construct a second line $b_2(t)$ with the same slope $s$ starting from $(t_1, \beta(t_1))$, and simulate the FPT $t_2$ of a Brownian motion starting from this point to $b_2(t)$. Repeat this process to construct lines $b_j(t)$ and corresponding times $t_j$ until some $t_c > T_{i+1}$ is reached. Here, the random point generated by this construction is denoted by $(t_c, \beta_c(t_c))$.
\end{enumerate}
\noindent If at any step the distance between $\beta(t_j)$ and $b_j(t_j)$ is very small before the jump time, it indicates that the Brownian path has crossed the threshold $\beta$ before $T_{i+1}$. In that case, discard the current trajectory and return to Step 2 with a new Brownian sample path by finding a new $t_1$.\\
\noindent This iterative scheme tracks the Brownian path using the hitting times to linear thresholds that lie strictly below the original threshold $\beta$, by construction of their slopes. This ensures that the entire trajectory remains below $\beta(t)$ up to time $T_{i+1}$. We summarise the full procedure in Algorithm~\ref{alg:point_conditional_Brownian}. Further, an illustration of this construction is provided in Figure~\ref{fig:cases_alg_cond_BM} and Figure~\ref{fig:ideal_alg_cond_BM}.
\begin{figure}[H]
    \centering
    \begin{subfigure}[t]{0.46\textwidth}
        \centering
        \includegraphics[width=\linewidth]{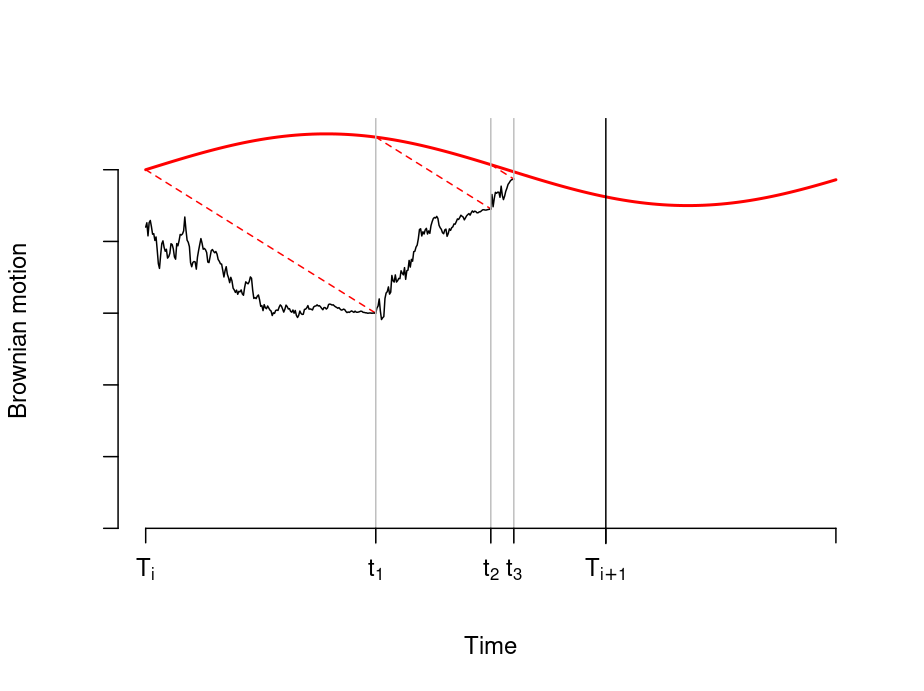}
        \caption{The distance between the constructed linear threshold and the original threshold $\beta$ is very small at $t_3$ before $T_{i+1}$. We generate a new $t_1$.}
        \label{fig:sub1_alg_cond_BM}
    \end{subfigure}
    \hfill
    \begin{subfigure}[t]{0.46\textwidth}
        \centering
        \includegraphics[width=\linewidth]{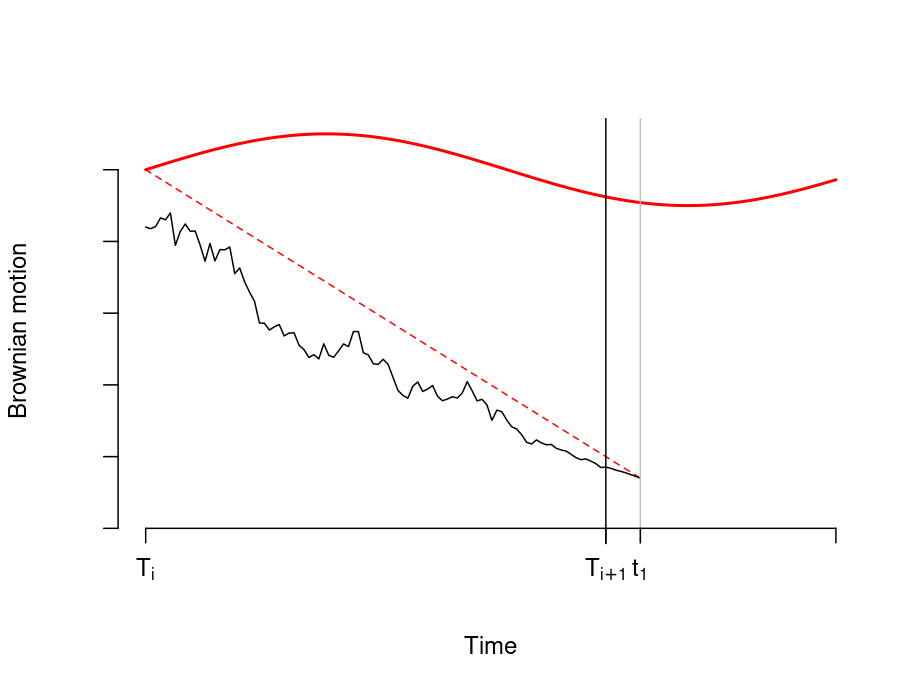}
        \caption{If $t_1 \geq T_{i+1}$, we reduce the slope and resample $t_1$ for the linear threshold with the new slope.}
        \label{fig:sub2_alg_cond_BM}
    \end{subfigure}
    \caption{Cases requiring resampling during the construction of $\big(t_c, \, \beta_c(t_c)\big)$.}
    \label{fig:cases_alg_cond_BM}
\end{figure}

\noindent In Figure~\ref{fig:cases_alg_cond_BM} depicts the two failure scenarios that can be encountered during the construction of the point $(t_c, \beta_c(t_c))$. In the first panel, the constructed line at its hitting time is too close to $\beta$ before the next jump time $T_{i+1}$, indicating that the Brownian path likely hit the original threshold $\beta$ before $T_{i+1}$, which is against the condition of the desired process \eqref{eq:aux_process}. Thus we resimulate $t_1$ for a new Brownian path.\\
In the second panel, the hitting $t_1$ to the constructed linear threshold exceeds $T_{i+1}$. Fixing a single value for slope $s$ would correspond to a fixed linear threshold $b_1$ as the intercept is initially fixed as $\beta(T_i)$. Choosing $t_c \gets t_1$ in this case would implicitly mean assuming that the hitting time to this fixed threshold $b_1$ is finite. To avoid this bias, we adapt $b_1$ based on whether $t_1 \geq T_{i+1}$.
\begin{figure}[H]
    \centering
    \includegraphics[width=0.55\textwidth]{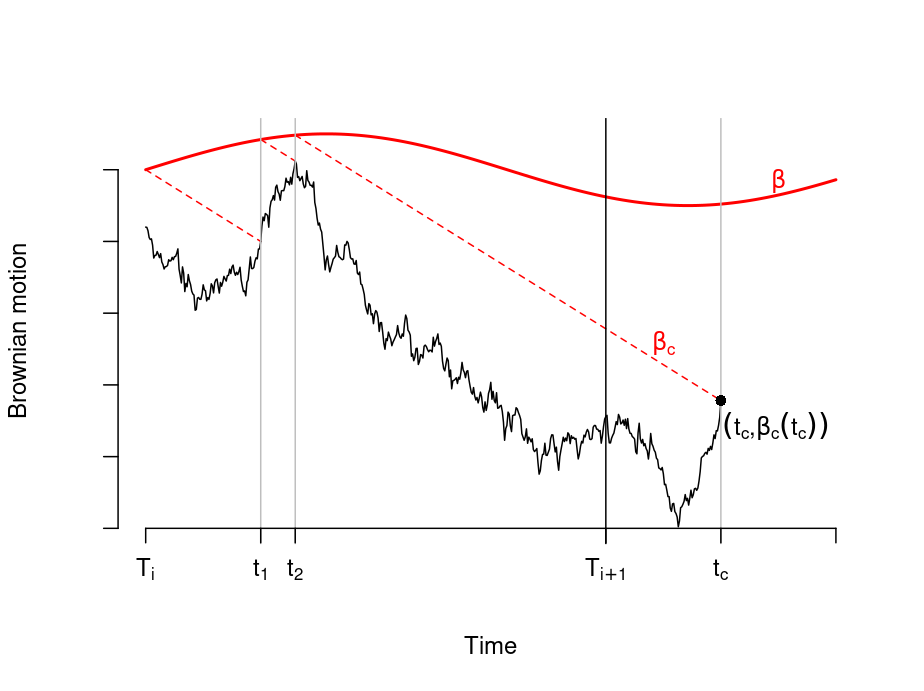}
    \caption{Ideal scenario: a valid $t_j > T_{i+1}$ is found without triggering the cases in Figure~\ref{fig:cases_alg_cond_BM}. We accept $t_c \gets t_j$.}
    \label{fig:ideal_alg_cond_BM}
\end{figure}
\noindent Figure~\ref{fig:ideal_alg_cond_BM} illustrates a scenario that validates the feasibility of our linear threshold based construction scheme. It shows that, under an appropriate choice of slope, the Brownian path can be guided to stay below the threshold without violating the conditioning constraint. Unlike the case described in Figure \ref{fig:sub2_alg_cond_BM}, this scenario would not have any bias as the intercept for $\beta_c$ is not fixed; instead the intercepts of the linear thresholds were constructed based on the path.

\begin{algorithm}[h]
\caption{Simulation of a point from a Brownian path conditioned on $\tau^{*}\geq T_{i+1}$}
\label{alg:point_conditional_Brownian}
\begin{algorithmic}[1]
\State \textbf{Input:} initial slope $s < \inf_{t \geq T_i} \beta'(t)$, minimum slope $s_{\text{min}}$, decrement $s_1$, tolerance $\epsilon$, initial value $X^{i}_{T_i}$
\vspace{0.2em}
\Repeat
  \State simulate $t_1$ FPT of Brownian motion to $b_1(t) = s(t - T_i) + \beta(T_i)$ starting at $T_i$
  \If{$t_1 \geq T_{i+1}$}
    \If{$s > s_{\text{min}}$}
      \State decrease slope: $s \gets s - s_1$
    \Else
      \State set $t_c \gets t_1$, $x \gets b_1(t_1)$
      \State \textbf{break} \Comment{accept path and exit}
    \EndIf
  \Else
    \State compute the distance: $d \gets \beta(t_1) - b_1(t_1)$
    \If{$d > \epsilon$}
      \State $j \gets 2$
      \While{$d > \epsilon$}
        \State simulate $t_j$ FPT of Brownian motion to $b_j(t) = s(t - t_{j-1}) + \beta(t_{j-1})$ starting from $t_{j-1}$
        \State $d \gets \beta(t_j) - b_j(t_j)$
        \If{$t_j \geq T_{i+1}$}
          \State \textbf{break} \Comment{exit while-loop}
        \Else
          \State $j \gets j + 1$
        \EndIf
      \EndWhile
      \If{$t_j \geq T_{i+1}$}
        \State set $t_c \gets t_j$, $x \gets b_j(t_j)$
        \State \textbf{break} \Comment{accept path and exit}
      \EndIf
    \EndIf
  \EndIf
\Until{false}
\State \textbf{Output:} $(t_c, x)$, where $x = \beta_c(t_c)$
\end{algorithmic}
\end{algorithm}
\noindent To ensure the well-posedness of the Algorithm~\ref{alg:point_conditional_Brownian}, it is essential to verify that the constructed point $(t_c, \beta_c(t_c))$ indeed belongs to a Brownian path conditioned to stay below the threshold $\beta$ until time $T_{i+1}$. The following theorem establishes the validity of the Algorithm.
\begin{theorem}
\label{thm:valid_alg_2}
Assume that the slope parameter satisfies $s_{\text{min}} = -\infty$. Then, as the tolerance $\epsilon \to 0$, the output $(t_c, \beta_c(t_c))$ of Algorithm~\ref{alg:point_conditional_Brownian} corresponds to a point sampled from a path of Brownian motion conditioned to stay below the threshold up to time $T_{i+1}$, as presented in Equation \eqref{eq:aux_process}.
\end{theorem}

\begin{proof}
The Algorithm~\ref{alg:point_conditional_Brownian} is designed to generate a point $(t_c, \beta_c(t_c))$ from a Brownian path that remains below the threshold $\beta(t)$ up to time $T_{i+1}$. More precisely, it constructs a sequence of linear thresholds $b_j(t)$ with slope $s < \inf_t \beta'(t)$ and simulates the FPTs $t_j$ to these lines, starting from $(t_{j-1}, \beta(t_{j-1}))$, with $t_0 = T_i$. This procedure continues until a time $t_j \geq T_{i+1}$ is reached, and the corresponding linear barrier $b_j$ is denoted by $\beta_c$. The final output is the point $(t_c, \beta_c(t_c)) := (t_j, b_j(t_j))$. Therefore, we aim to show here that in the limit as $\epsilon \to 0$ and $s_{\text{min}} = -\infty$, the law of the process 
\begin{equation*}
 \{W_t \mid t \in [T_i, t_c],\ W_{t_c} = \beta_c(t_c),\ t_c \geq T_{i+1} \}   
\end{equation*}
coincides with that of the Brownian motion conditioned to stay below $\beta$ until $T_{i+1}$, i.e., the process defined in Equation \eqref{eq:aux_process}.
By the definition of the accepted point $t_c$, we have that $t_c \geq T_{i+1}$ and $W_t < \beta(t)$ for all $t \in [T_i, t_c)$. Therefore, the simulated paths satisfy $\tau^{*} > T_{i+1}$ and hence we have that
\begin{equation*}
  \text{sample paths of } \{W_t \mid t_c \geq T_{i+1},\ W_{t_c} = \beta_c(t_c)\}
\subseteq
\text{sample paths of } \{W_t \mid \tau^* \geq T_{i+1}\}.  
\end{equation*}
Now, consider the probability,
\begin{equation*}
 \mathbb{P}\{W_t \mid t_c \geq T_{i+1} ,\beta_c(t_c) = W_{t_c}\}.    
\end{equation*}
The pair $(t_j,b_j)$ for $j\geq 2$ is constructed by tracking the Brownian path such that each $t_j$ satisfies $b_{j}(t_j)=W_{t_j}$. Thus, $\beta_c(t_c)=W_{t_c}$ is a sure event except if we choose $t_c=t_1$, in which case $b_1$ is fixed. Since we start with the fixed slope $s$ and a fixed initial point $\beta(T_i)$ for $b_1$, assuming that FPT to $b_1$ is finite would introduce bias. To avoid this, the algorithm reduces the slope if $t_1\geq T_{i+1}$, repeating this step as needed.\\
Therefore if we do not impose any restriction on the minimum slope, i.e. if $s_\text{min}=-\infty$, this ensures that $t_1<T_{i+1}$ and hence
\[
\beta_c(t_c)=W_{t_c}
\]
is a sure event. Consequently we can write:
\[
\mathbb{P}\{W_t \mid t_c \geq T_{i+1} ,\beta_c(t_c) = W_{t_c}\}=\mathbb{P}\{W_t \mid t_c \geq T_{i+1}\}.
\]
The probability on the right describes the path of Brownian motion which stays below the threshold $\beta$ at least until time $t_c$. On the other hand, the process defined in \eqref{eq:aux_process} corresponds to Brownian motion conditioned to stay below $\beta$ until $\tau^*$. If the crossing from $\beta$ is detected exactly, then both probabilities are equivalent on the interval $[T_i,t_c]$, i.e.
\[
  \mathbb{P}\{W_t \mid t_c \geq T_{i+1}\} = \mathbb{P}\left\{ W_t \,\middle|\, \tau^{*} \geq T_{i+1} \right\} \qquad t\in[T_i,t_c].
\]
Accurate detection of threshold crossing is essential. For instance, if a large $\epsilon$ is chosen, the algorithm may falsely reject paths that actually remain below $\beta$ until $T_{i+1}$. These paths would be valid under the law of the auxiliary process $\{W_t \mid \tau^*\geq T_{i+1}\}$ but would be wrongly excluded by the algorithm. Hence, precise detection of crossings is crucial to ensure convergence to the correct conditional law.
\end{proof}

\noindent Now that a valid point $(t_c, \beta_c(t_c))$ has been sampled, the next step is to simulate the entire Brownian trajectory on the time interval $[T_i, t_c]$ that stays below the threshold $\beta$ and passes exactly through this point $(t_c, \beta_c(t_c))$.\\
\noindent To construct such a path, we use the fact that the following process
\begin{equation*}
  R_h = \beta(t_c - h) - W_{t_c - h}, \qquad 0 \leq h \leq t_c - T_i,  
\end{equation*}
is a Bessel bridge. We start by simulating a Bessel bridge $R_h$ on $[0, t_c - T_i]$ starting at
\begin{equation*}
R_0 = \beta(t_c) - \beta_c(t_c),    
\end{equation*}
and ending at
\begin{equation*}
 R_{t_c - T_i} = \beta(T_i) - X^{i,\infty}_{T_i}.   
\end{equation*}
After generating the Bessel bridge $R_h$, we invert the transformation to recover the Brownian path $W_t$ over $[T_i, t_c]$. This path will be our \emph{auxiliary process} for the rejection sampling algorithm.
\begin{remark}
One might consider directly using the FPT $\tau_{\beta}^{\text{cont}, i}$ and generating a Bessel bridge on $[T_i, \tau_{\beta}^{\text{cont}, i}]$ starting at $0$ and ending at $\beta(T_i) - W_{T_i}$. However, this would lead to biased sampling because it excludes Brownian paths that remain strictly below the threshold and never hit it.
\end{remark}
\noindent We now turn to the key ingredient of the Exact simulation method: the \emph{rejection probability}. While we are ultimately interested only in the value of the process at time $T_{i+1}$, the acceptance-rejection step must be performed on the entire path over $[T_i, t_c]$, since the law of the auxiliary process depends on its endpoint at $t_c$.\\
\noindent Following again the Girsanov-based framework introduced in \cite{beskos2005exact,beskos2006retrospective,herrmann2019exact}, the acceptance probability is derived from a change of measure between the diffusion process and the auxiliary Brownian motion. In this setting, for any bounded measurable functional $\Psi$, we have
\begin{equation*}
\mathbb{E}_{\mathbb{P}}\left[\Psi\left(X^{i,\infty}_t\right)_{T_i \leq t \leq t_c} \right] = 
\mathbb{E}_{\mathbb{Q}}\left[\Psi\left(\tilde{X}^i_t\right)_{T_i \leq t \leq t_c} \cdot \eta^{[2]}(t)\right],    
\end{equation*}
where
\begin{equation}
\eta^{[2]}(t) = \exp\left(
A\left(t_c,\,( \tilde{X}^i_{t_c}, z_i)\right)
- \int_{T_i}^{t_c} \gamma_2\left(t,\, (\tilde{X}^i_t, z_i)\right) \, dt
\right)
\end{equation}
is the rejection probability.\\
\noindent
As discussed in Subsection~\ref{subsection3.2}, directly evaluating the acceptance probability $\eta^{[2]}(t)$ is not feasible in closed form. Therefore, we employ a Poisson thinning representation that allows the acceptance or rejection of a path based on finitely many sampled time points over the interval $[T_i, t_c]$. However, with this approach it is not guaranteed that one of the sampled points coincides exactly with the jump time $T_{i+1}$, which is essential in our context.\\
\noindent To address this, we decompose the Girsanov weight in the interval $[T_i, t_c]$ as follows. First, we have
\begin{align*}
\mathbb{E}_{\mathbb{P}}\left[\Psi\left(X_t^{i,\infty}\right)_{T_i \leq t \leq t_c}\right]
= \mathbb{E}_{\mathbb{Q}}\Bigg[
\Psi\left(\tilde{X}^i_t\right) \exp\Bigg(
\int_{T_i}^{t_c} \alpha\big(t,\,(\tilde{X}^i_t, z_i)\big) \, dB_t 
- \frac{1}{2} \int_{T_i}^{t_c} \alpha^2\big(t,\, ( \tilde{X}^i_t, z_i)\big) \, dt
\Bigg)
\Bigg].
\end{align*}
\noindent We then split the integrals at $T_{i+1}$ as follows
\begin{align*}
\mathbb{E}_{\mathbb{P}}\left[\Psi\left(X_t^{i,\infty}\right)_{T_i \leq t \leq t_c}\right]
= \mathbb{E}_{\mathbb{Q}}\Bigg[
\Psi\left(\tilde{X}^i_t\right) \exp\Bigg(
&\int_{T_i}^{T_{i+1}} \alpha\big(t,\, (\tilde{X}^i_t, z_i)\big) \, dB_t 
- \frac{1}{2} \int_{T_i}^{T_{i+1}} \alpha^2\big(t,\, (\tilde{X}^i_t, z_i)\big) \, dt \\[0.22cm]
+ &\int_{T_{i+1}}^{t_c} \alpha\big(t,\, ( \tilde{X}^i_t, z_i)\big) \, dB_t 
- \frac{1}{2} \int_{T_{i+1}}^{t_c} \alpha^2\big(t,\, ( \tilde{X}^i_t, z_i)\big) \, dt
\Bigg)
\Bigg].
\end{align*}
\noindent Applying Itô's formula to the function $A(t,\,(x, z_i))$ yields the following decomposition:
\begin{align}
\mathbb{E}_{\mathbb{P}}\left[\Psi\left(X_t^{i,\infty}\right)\right]
= & \; \mathbb{E}_{\mathbb{Q}}\left[\Psi\left(\tilde{X}^i_t\right)\right]
\cdot \exp\Big({-A\big(T_i,\,( \tilde{X}^i_{T_i}, z_i)\big)}\Big) \cdot
\mathbb{E}_{\mathbb{Q}}\left[\exp\left(-\int_{T_i}^{T_{i+1}} \gamma_2\big(t,\, (\tilde{X}^i_t, z_i)\big) \, dt \right)\right] \notag \\[0.31cm]
& \cdot \mathbb{E}_{\mathbb{Q}}\left[\exp\Big({A\big(T_{i+1},\, (\tilde{X}^i_{T_{i+1}}, z_i)\big)}\Big)\right] 
\cdot \exp\Big({-A\big(T_{i+1}, \, (\tilde{X}^i_{T_{i+1}}, z_i)\big)}\Big) \notag \\[0.3cm]
& \cdot \mathbb{E}_{\mathbb{Q}}\left[\exp\left(-\int_{T_{i+1}}^{t_c} \gamma_2\big(t,\, (\tilde{X}^i_t, z_i)\big) \, dt \right)\right] \cdot \mathbb{E}_{\mathbb{Q}}\left[\exp\Big({A\big(t_c,\,( \tilde{X}^i_{t_c}, z_i)\big)}\Big)\right].
\end{align}
\noindent This decomposition shows that the overall rejection probability can be broken down into four distinct components, each of which must be evaluated separately.
\begin{itemize}
    \item[1.] Rejection probability on $[T_i, T_{i+1}]$:
    \begin{equation*}
        \mathbb{E}_{\mathbb{Q}}\left[\exp\left(-\int_{T_i}^{T_{i+1}} \gamma_2\big(t,\, (\tilde{X}^i_{t}, z_i)\big) \, dt\right)\right].
    \end{equation*}
    This term can be computed using Poisson thinning, as in \cite{beskos2006retrospective}, applying Theorem~\ref{thm:poisson_thinning} with substitutions $\gamma_1 + \gamma_2 \mapsto \gamma_2$, $\kappa \mapsto \kappa_2$, and $\mathcal{T}_\beta \mapsto T_{i+1}$, where $\kappa_2 \geq \sup \gamma_2$.

    \item[2.] Acceptance weight at time $T_{i+1}$:
    \begin{equation*}
\mathbb{E}_{\mathbb{Q}}\left[\exp\Big({A\big(T_{i+1},\, (\tilde{X}^i_{T_{i+1}}, z_i)\big)}\Big)\right].
    \end{equation*}

    \item[3.] Rejection probability on $[T_{i+1}, t_c]$:
    \begin{equation*}
        \mathbb{E}_{\mathbb{Q}}\left[\exp\left(-\int_{T_{i+1}}^{t_c}  \gamma_2\big(t,\, (\tilde{X}^i_{t}, z_i)\big) \, dt\right)\right],
    \end{equation*}
    which is again evaluated via thinning on the interval $[T_{i+1}, t_c]$.

    \item[4.] Acceptance weight at final time $t_c$:
    \begin{equation*}
     \mathbb{E}_{\mathbb{Q}}\left[\exp\Big({A\big(t_{c},\, (\tilde{X}^i_{t_{c}}, z_i)\big)}\Big)\right].   
    \end{equation*}
\end{itemize}

\noindent All four components must yield sufficiently high values to accept the auxiliary process \eqref{eq:aux_process}. If any one of them results in a rejection, a new candidate point $(t_c, \beta_c(t_c))$ is generated via Algorithm~\ref{alg:point_conditional_Brownian}, and the process is restarted with a new Brownian trajectory.\\
\noindent To ensure the well-posedness of the rejection probability and the validity of the thinning method, we impose the following conditions.

\begin{assumption}
\label{ass:kappa2andA+}
\begin{enumerate}[label=(\roman*)]
    \item The function $\gamma_2$ is uniformly bounded:
    \begin{equation*}
        0 \leq \gamma_2\left(t, (x, z_i)\right) \leq \kappa_2.
    \end{equation*}

    \item The function $A(t,\,(x,z_i))$ is bounded above by a constant $A^+ > 0$.
\end{enumerate}
\end{assumption}
\begin{algorithm}[H]
\caption{Exact simulation of the diffusion $X^{i,\infty}$ at $T_{i+1}$ conditioned on $\tau^{*}\geq T_{i+1}$}
\label{alg:trajectory of continuous part until jump}
\begin{algorithmic}[1]
\State \textbf{Input:} $(t_c,\beta_c(t_c)) \sim$ Algorithm~\ref{alg:point_conditional_Brownian}, \quad $\kappa_2 \geq \sup_{(t,x)} \gamma_2(t,\,(x,z))$ and $X^{i,\infty}_{T_i}$
\vspace{0.5em}
\Function{Rejection-Sampling}{$t_{\text{start}}, t_{\text{end}}, x_{\text{start}}$}
    \State initialise: $l \gets c\big((\beta(t_c) - \beta_c(t_c)),\, 0, 0\big)$, \quad $\mathcal{E}_0 \gets 0$, \quad $\mathcal{E}_1 \gets \text{Exp}(\kappa_2)$, \quad $\mathcal{W} \gets 0$
    \While{$\mathcal{E}_1 < \big(t_{\text{end}} - t_{\text{start}}\big)$ and $\mathcal{W} = 0$}
        \State Sample $G \sim \mathcal{N}(0, I_3)$, \quad $e \sim \text{Exp}(\kappa_2)$, \quad $\mathcal{U} \sim \mathcal{U}(0,1)$
        \State update:
        \begin{equation*}
        l \gets \frac{(t_c - t_{\text{start}}) - \mathcal{E}_1}{(t_c - t_{\text{start}}) - \mathcal{E}_0} \cdot l + \sqrt{\frac{((t_c - t_{\text{start}}) - \mathcal{E}_1)(\mathcal{E}_1 - \mathcal{E}_0)}{(t_c - t_{\text{start}}) - \mathcal{E}_0}} \cdot G   \end{equation*}
        \begin{equation*}
            \xi \gets \beta(\mathcal{E}_1 + t_{\text{start}}) - \left\| \dfrac{\mathcal{E}_1(\beta(t_{\text{start}}) - x_{\text{start}})}{t_c - t_{\text{start}}} \right\|
        \end{equation*}

        \If{$\kappa_2 \cdot \mathcal{U} \leq \gamma_2\big((\mathcal{E}_1 + t_{\text{start}}),\,(\xi,z_i)\big)$}
            \State $\mathcal{W} \gets 1$
        \Else
            \State $\mathcal{E}_0 \gets \mathcal{E}_1$, \quad $\mathcal{E}_1 \gets \mathcal{E}_1 + e$
        \EndIf
    \EndWhile
    \State \Return $(\mathcal{W}, \mathcal{E}_1, l)$
\EndFunction

\vspace{0.7em}
\State $(\mathcal{W}, \mathcal{E}_1, l) \gets$ \Call{Rejection-Sampling}{$T_i$, $T_{i+1}$, $X^{i,\infty}_{T_i}$} \Comment{sampling over $[T_i, T_{i+1}]$}
\If{$\mathcal{W} = 0$ and $\mathcal{E}_1 \geq (T_{i+1} - T_i)$}
\vspace{0.4em}
    \State $x_j \gets \beta(T_{i+1}) - \left\| \dfrac{(T_{i+1} - T_i)(\beta(T_i) - X^{i,\infty}_{T_i})}{t_c - T_i} \right\|$
\vspace{0.3em}    
    \State Sample $\mathcal{U} \sim \mathcal{U}(0,1)$
    \If{$\mathcal{U} \cdot e^{A^+} > e^{A\big(T_{i+1}, (x_j, z_i)\big)}$} \Comment{evaluation at $T_{i+1}$}
        \State $\mathcal{W} \gets 1$
        \State restart with a new $(t_c, \beta_c(t_c))$
    \EndIf
\Else
    \State restart with a new $(t_c, \beta_c(t_c))$
\EndIf

\If{$\mathcal{W} = 0$}
    \State $(\mathcal{W}, \mathcal{E}_1, l) \gets$ \Call{Rejection-Sampling}{$T_{i+1}$, $t_c$, $x_j$} \Comment{sampling over $[T_{i+1}, t_c]$}
    \If{$\mathcal{W} = 0$ and $\mathcal{E}_1 \geq t_c - T_{i+1}$}
        \State sample $\mathcal{U} \sim \mathcal{U}(0,1)$
        \If{$\mathcal{U} \cdot e^{A^+} > e^{A\big(t_c,\,( \beta_c(t_c), z_i)\big)}$} \Comment{evaluation at $t_c$}
            \State $\mathcal{W} \gets 1$
            \State restart with a new $(t_c, \beta_c(t_c))$
        \EndIf
    \EndIf
\EndIf
\State \textbf{Output:} $x_j$
\end{algorithmic}
\end{algorithm}

\subsection{Exact algorithm for the FPT of PDifMPs}
We now combine the components developed in the previous sections into a unified scheme for the exact simulation of the FPT of a PDifMP to a time-dependent threshold. This algorithm extends the diffusion-based methods of \cite{khurana2024exact,herrmann2019exact} and the jump-diffusion framework of \cite{herrmann2023exact} to the more general setting of PDifMPs.
\noindent The procedure is built on a piecewise recursive strategy that alternates between continuous diffusion phases and discrete jump updates. The core simulation tasks are:
\begin{itemize}
    \item[(i)] sampling the FPT $\tau_{\beta}^{\text{cont}, i}$ of the auxiliary diffusion $X^{i,\infty}$, using a Girsanov-based Exact method adapted to time- and space-dependent SDEs, see Subsection~\ref{subsection3.2}.
    
    \item[(ii)] simulating the diffusion value at the next jump time, conditioned on no threshold crossing before the jump. This requires constructing an auxiliary process and a tailored rejection step to account for the conditional structure and time-dependent threshold, see Subsection~\ref{subsection3.3}.
\end{itemize}
\noindent Bringing all of these steps together yields an Exact algorithm for simulating the FPT of the PDifMP $(U_t)_{t\geq 0}$. To ensure termination even when the process does not reach the threshold in finite time, in the following Algorithm~\ref{alg:final algorithm} we introduce a fixed terminal time $T_f$. The output is then the FPT if it occurs before $T_f$, and $T_f$ otherwise.
\begin{algorithm}[H]
\caption{Exact simulation of the FPT with terminal time $T_f$}
\label{alg:final algorithm}
\begin{algorithmic}[1]
\State Initialise: terminal time $T_f$, jump index $i \gets 0$, $T_0 \gets 0$, $X^{i,\infty}_{T_{i}}=x_i$, \texttt{crossing} $\gets$ \texttt{FALSE}
\Repeat
    \State sample exponential waiting time $\omega \sim \text{Exp}(\lambda)$ and set $T_{i+1} \gets T_i + \omega$
    \If{$T_i \geq T_f$} 
        \State \texttt{crossing} $\gets$ \texttt{TRUE}, \quad $\tau \gets T_f$
    \Else
        \State simulate $\tau_1 \sim$ Algorithm \ref{alg:fpt_simulation}
        \If{$T_f \leq T_{i+1}$ \textbf{and} $\tau_1 > T_f$}
            \State \texttt{crossing} $\gets$ \texttt{TRUE}, \quad $\tau \gets T_f$
        \EndIf
        \If{$\tau_1 < \min(T_{i+1}, T_f)$}
            \State \texttt{crossing} $\gets$ \texttt{TRUE}, \quad $\tau \gets \tau_1$
        \EndIf  
        \If{$T_f > T_{i+1} \text{ and } \tau_1 \geq T_{i+1}$}
        \Comment{ FPT did not occur before the next jump. Simulate the process at $T_{i+1}$}
            \If{$\tau_1 = T_{i+1}$}
                \State $x_j \gets \beta(T_{i+1})$
            \Else    
                \State $x_j \sim$ Algorithm \ref{alg:trajectory of continuous part until jump}
            \EndIf 
            \State sample $z_i\sim \mathcal{Q}$
             \State compute tracking state: $Y^{i,\infty}_{T_{i+1}} \gets {F^i}^{-1}\big(T_{i+1},\, (x_j, z_i)\big)$
             \vspace{0.2em}
            \State add jump: $Y^{i+1}_{T_{i+1}} \gets Y^{i,\infty}_{T_{i+1}} + j\big(T_{i+1},\, Y^{i,\infty}_{T_{i+1}}, z_i\big)$
            \vspace{0.25em}
            \State transform: $X^{{i+1},\infty}_{T_{i+1}} \gets F^{i+1}\big(T_{i+1},\, (Y^{i+1}_{T_{i+1}}, z_i)\big)$
            \vspace{0.2em}
            \If{$X^{i+1,\infty}_{T_{i+1}} \geq \beta(T_{i+1})$}
                \State \texttt{crossing} $\gets$ \texttt{TRUE}, \quad $\tau \gets T_{i+1}$
            \Else  \qquad $i\gets i+1$    
            \EndIf    
        \EndIf    
    \EndIf
\Until{\texttt{crossing} = \texttt{TRUE}}
\State \textbf{Output:} $\tau$
\end{algorithmic}
\end{algorithm}
\begin{remark}
At each jump time $T_{i+1}$, the continuous component of the PDifMP undergoes a discrete jump. Since the auxiliary process $X^{i,\infty}$ evolves without jumps, we first recover the pre-jump value of the underlying diffusion $Y^{i,\infty}_{T_{i+1}}$ via the inverse Lamperti transform. The jump size is then added to obtain the post-jump value $Y^{i+1}_{T_{i+1}}$, which serves as the initial condition for the next inter-jump interval. Finally, we reapply the Lamperti transform to define $X^{i+1,\infty}$. This ensures the correct initialisation of the next interval in the piecewise simulation.
\end{remark}
\begin{theorem}
Assume that Assumptions \ref{ass: on drift}–\ref{ass:kappa2andA+} hold, and let $s_{\text{min}} = -\infty$. Then, as $\epsilon \to 0$, the distribution of the output of Algorithm~\ref{alg:final algorithm}, i.e. $\tau$, converges in law to the distribution of $\min( \tau_{\tilde{\beta}}, T_f)$, where $\tau_{\tilde{\beta}}$ denotes the FPT of the PDifMP to the threshold $\tilde{\beta}$.
\label{final_thm}
\end{theorem}
\begin{proof}
Since Algorithm~\ref{alg:final algorithm} is a coupling of several steps, the proof of Theorem~\ref{final_thm} follows by combining already well established results. We outline the main arguments here.

\begin{enumerate}[label=(\roman*)]
    \item \emph{Exact FPT sampling.} The simulation of the FPT $\tau_{\beta}^{\text{cont}, i}$ of the tracking diffusion $X^{i,\infty}$ is exact, based on the method introduced in \cite{khurana2024exact, herrmann2019exact}, adapted here to account for time- and space-dependent coefficients via a modified rejection probability.

    \item \emph{Conditional value at jump time.} The simulation of $X^{i,\infty}_{T_{i+1}}$ conditioned on $\tau_{\beta}^{\text{cont}, i} \geq T_{i+1}$ is treated using an auxiliary process and a tailored acceptance-rejection strategy. As shown in Theorem~\ref{thm:valid_alg_2}, the distribution of this approximation converges to the exact conditional law as $\epsilon \to 0$ and $s_{\text{min}} \to -\infty$.

    \item \emph{Recursive structure.} The combination of components (i) and (ii) within the hybrid recursive framework follows the structure presented in \cite{herrmann2023exact}, ensuring that the overall algorithm remains consistent with the target distribution of the FPT.
\end{enumerate}
\noindent Therefore, it is straightforward to show that Algorithm~\ref{alg:final algorithm} produces samples whose distribution converges to that of $\tau_\beta \wedge T_f$ in the limit $\epsilon \to 0$. \qedhere
\end{proof}

\section{Numerical examples}
\label{Numa_exp}

In this section, we apply Algorithm~\ref{alg:final algorithm} to simulate the FPT of a PDifMP in two illustrative examples. To evaluate the accuracy of our method, we compare the results with those obtained using the classical Euler-Maruyama scheme with a fine time discretisation. We present the estimated FPT densities for both approaches and perform statistical tests to  quantify how close the two methods are.

\begin{remark}
Throughout the examples, we assume a constant jump rate $\lambda$. This guarantees the integrability condition~\eqref{eq:lambda_integrability} and the finite-jump condition~\eqref{jump_count}. The results can be extended to time dependent $\lambda$.
\end{remark}

\noindent Algorithm~\ref{alg:final algorithm} depends on two parameters, $\epsilon$ and $s_{\text{min}}$, which influence the accuracy of the simulated FPT. In the examples below, we examine the effect of $s_{\text{min}}$ by varying its value and observing its impact on the resulting density estimates and test outcomes. This complements the theoretical analysis established in Theorem~\ref{thm:valid_alg_2}.

\begin{remark}
We do not explicitly study the impact of the parameter $\epsilon$, as previous works, see for instance \cite{herrmann2016first, khurana2024exact}, already demonstrates that the algorithm accurately detects threshold crossings in the limit $\epsilon \to 0$. Consequently, taking $\epsilon$ sufficiently small guarantees precision in the hitting time estimates.
\end{remark}

\subsection*{Example 1: time-homogeneous PDifMP}
As a first example, we consider a time-homogeneous PDifMP $U_t = (Y_t, Z_t)$ defined by the following dynamics. Let $T_f$ be a deterministic final time, for all $t$ in the inter-jump time intervals $\in [T_{i}, T_{i+1})$, the continuous dynamics of the process $U_t$ are given by

\begin{equation}
   \left\{
\begin{array}{ll}
  dY^{i}_{t} = \big(1.6 + \sin(Y^{i}_{t})\big)\,dt + dB_t, & \quad t \in [T_{i}, T_{i+1}), \\[0.65em]
  dZ^{i}_{t} = 0, & 
\end{array}
\right.
\label{first_exp}
\end{equation}
with initial condition $Y^0_0 = -1$. At jump times $T_{i+1}$ the process is updated as follows:

\begin{equation*}
   \left\{
\begin{array}{ll}
  Y^{i+1}_{T_{i+1}} = Y^{i}_{T_{i+1}} - z_i \sin(Y^{i}_{T_{i+1}}), & \\[0.65em]
  Z^{i+1}_{T_{i+1}} = z_i, & \quad \text{with } z_i \sim \text{Exp}(1)
\end{array}
\right.
\end{equation*}
The characteristic triple $(\phi, \lambda, \mathcal{Q})$ of the PDifMP $U_t$ is given by
\begin{equation}
	\label{charc}
	\left\{
	\begin{array}{ll}
		\phi &= \text{consecutive solutions of the system \eqref{first_exp} on } [T_i, T_{i+1}), \\[0.2cm]
		\lambda(u) &= \lambda, \quad \text{constant jump rate}, \\[0.2cm]
		\mathcal{Q}((y, z), \mathcal{A}) &= \mathbb{P}\left(\left(y - \zeta \sin(y),\, \zeta\right) \in \mathcal{A} \right), \quad \zeta \sim \text{Exp}(1).

	\end{array}
	\right.
\end{equation}

\noindent The time-dependent threshold is defined by the linear function:
\begin{equation*}
  \beta(t) = -t + 1.  
\end{equation*}

\noindent We start by verifying that the PDifMP $U_t$ satisfies the structural assumptions required for the simulation algorithm:
\begin{itemize}
    \item The drift and diffusion coefficients in Equation\eqref{first_exp} are differentiable and bounded, hence satisfy both the Lipschitz continuity and linear growth conditions.
    \item The jump function $j(y,z) = -z\sin(y)$ is bounded considering $z$ is a constant and thus satisfies Assumption~\ref{ass:jump-size}.
\end{itemize}
\noindent We define the tracking process $Y^{i,\infty}$, used in the simulation of the FPT by
\begin{equation*}
 dY^{i,\infty}_t = \left(1.6 + \sin(Y^{i,\infty}_t)\right)\,dt + dB_t, \qquad t \in [T_i, \infty).   
\end{equation*}
To apply Girsanov’s theorem for the exact simulation of the tracking process, we verify that the drift term is regular and satisfies Novikov's condition. The drift is continuously differentiable and bounded, and satisfies the Novikov condition:
\begin{align*}
\mathbb{E}\left[\exp\left(\frac{1}{2} \int_0^{\tau^{*}} \left(1.6 + \sin(Y_s)\right)^2 ds\right)\right]
&= \mathbb{E}\left[\exp\left(\frac{1}{2} \int_0^{\tau^{*}} \left((1.6)^2 + \sin^2(Y_s) + 2 \cdot (1.6) \sin(Y_s)\right) ds\right)\right] \\[0.8em]
&\leq \mathbb{E}\left[\exp\left(36 \cdot \tau^{*}\right)\right].
\end{align*}
\noindent It is straightforward to see that the threshold function is differentiable. Additionally, it is necessary to ensure that the tracking process $Y^{i,\infty}$ does not explode before reaching the threshold. This follows from the fact that the drift is bounded, the diffusion coefficient is Lipschitz continuous, and its square is bounded away from zero. By Proposition 5.17 in \cite{karatzas1991brownian}, these conditions guarantee non-explosion of the solution to the SDE.\\
\noindent We now specify the functions $\gamma_1$, $\gamma_2$, and $A$ associated with this example, and verify that they satisfy the assumptions required for the application of the exact simulation method:
\begin{align*}
\gamma_1(t) &= 1.6 + \sin(-t + 1), \\[0.65em]
\gamma_2(t, (y, z_i)) &= \frac{(1.6 + \sin(y))^2 + \cos(y)}{2}, \\[0.65em]
A(t, (y, z_i)) &= 1.6x - \cos(y).
\end{align*}

\noindent The function $\gamma_1$ is bounded and strictly positive, taking values in the interval $[0.6, 2.6]$. Numerical evaluation confirms that $\gamma_2$ is also strictly positive and bounded above by approximately 3.88. Consequently, the sum $\gamma_1 + \gamma_2$ is non-negative and uniformly bounded.\\
\noindent Since the simulation terminates at the FPT $\tau_\beta$, it is sufficient that these assumptions hold up to that time. Because the process remains below the threshold $\beta(t) = -t + 1$, which is bounded above by 1, we have $Y^{i,\infty}_t < 1$ for all $t < \tau_\beta$. Therefore, on this interval, the function $A(t, (y, z_i))$ is also bounded above by $2.6$.

\noindent Finally, we apply Algorithm~\ref{alg:final algorithm} to simulate samples of $\min(\tau_\beta, T_f)$. To assess the accuracy of our method, we compare it against a high-resolution Euler–Maruyama method (i.e. using a very fine time discretisation). The corresponding densities results for different $s_{\text{min}}$ values are displayed in Figure~\ref{fig:example1_1}.

\begin{figure}[H]
    \centering
    \includegraphics[width=0.6\linewidth]{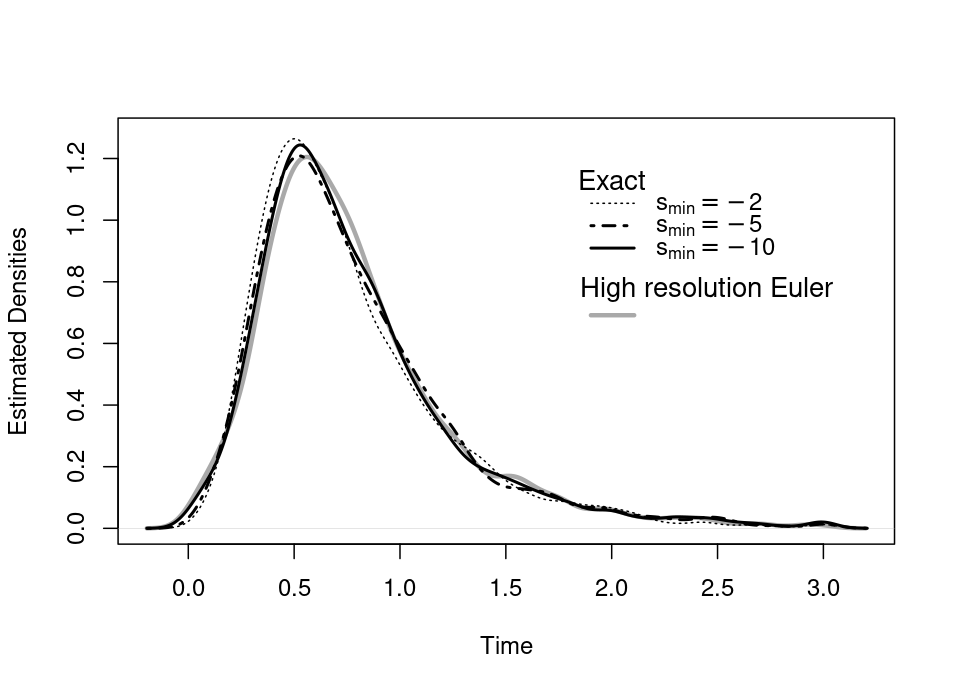}
    \caption{Density plots of FPT of the PDifMP in example 1, based on 3,000 simulations using two methods: (i) the Exact method for different values of the parameter $s_{\text{min}}$ and (ii) a high-resolution Euler–Maruyama scheme. The jump rate is fixed at $\lambda = 1$, and the terminal time is $T_f = 3$.}
    \label{fig:example1_1}
\end{figure}

\noindent The density curves produced by the Exact method show strong alignment with those from the Euler–Maruyama scheme, confirming the validity of the proposed algorithm. The choice of $s_{\text{min}}$ does not influence the accuracy much. In particular, the results for $s_{\text{min}} = -5$ and $s_{\text{min}} = -10$ are nearly indistinguishable from the benchmark.\\
To quantify the similarity between the simulated distributions, we perform a Kolmogorov–Smirnov (KS) test comparing the Exact and Euler–Maruyama outputs for various values of $s_{\text{min}}$:

\begin{table}[H]
\centering
\begin{tabular}{|c|c|c|c|}
\hline
$s_\text{min}$ & $-2$ & $-5$ & $-10$ \\
\hline
p-value  & $0.0004$  & $0.2242$  & $0.3369$  \\
\hline
\end{tabular}
\caption{KS test p-values comparing the Exact method with the Euler–Maruyama}
\label{tab:s_pvalue}
\end{table}

\noindent The results in Table~\ref{tab:s_pvalue} show that the p-values are comfortably above common significance levels. This provides empirical support for the theoretical convergence result established in Theorem~\ref{thm:valid_alg_2}, and confirms that moderately small values of $s_{\text{min}}$ already yield reliable approximations in practice.
\subsection*{Example 2: time-inhomogeneous PDifMP}

We consider now the following PDifMP $U_t = (Y_t, Z_t)$, with dynamics defined as follows. Between jumps, the system evolves according to the SDE:
\begin{equation}
\left\{
\begin{array}{ll}
dY_t^{i} = \left(Z_t^i + \frac{1}{2} \sin(t + Y_t^i)\right) dt + dB_t, & \quad t \in [T_i, T_{i+1}), \\[0.65em]
dZ_t^i = 0. &
\end{array}
\right.
\label{Exmpl_2}
\end{equation}
\noindent The process is initialised at $Y_0^0 = -1$. At each jump time $T_{i+1}$, the process is updated as:
\begin{equation*}
\left\{
\begin{array}{ll}
Y^{i+1}_{T_{i+1}} = Y^{i}_{T_{i+1}} + j(Y^{i}_{T_{i+1}}, z_i), \\[0.4em]
Z^{i+1}_{T_{i+1}} = z_i,
\end{array}
\right.
\end{equation*}
where $z_i \sim \mathcal{U}(1.8, 3)$ i.i.d., and the jump function is given by $j(y, z) = \frac{1 - y}{z}$.
\noindent We consider the same linear threshold as in Example 1. The PDifMP $U_t = (Y_t, Z_t)$ is defined by the following local characteristics:
\begin{equation*}
\label{eq:example2_char}
\left\{
\begin{array}{ll}
\phi &= \text{solution of system \eqref{Exmpl_2}} \\[0.4em]
\lambda(u) &= \lambda, \\[0.4em]
\mathcal{Q}((y, z), \mathcal{A}) &= \mathbb{P}\left( \left(y + \frac{1 - y}{z'},\, \zeta\right) \in \mathcal{A} \right), \quad \zeta \sim \mathcal{U}(1.8, 3).
\end{array}
\right.
\end{equation*}
\noindent We verify that the PDifMP satisfies the assumptions required for our simulation framework:
\begin{itemize}
    \item The drift and diffusion coefficients satisfy Lipschitz continuity and linear growth conditions.
    \item The jump size function $j(y,z)$ satisfies a linear growth condition in $y$. The drift and diffusion coefficients meet the conditions ensuring non-explosion of the process \cite{karatzas1991brownian}, and since $z$ is constant, the jump size is bounded.
\end{itemize}
\noindent Furthermore, the tracking diffusion $Y^{i,\infty}$ on the interval $[T_i, \infty)$ satisfies:
\begin{equation*}
dY^{i,\infty}_t = \left(z_i + \frac{1}{2} \sin(t + Y^{i,\infty}_t)\right) dt + dB_t.    
\end{equation*}
The drift term is continuously differentiable and satisfies Novikov’s condition. In particular, we have:
\begin{align*}
\mathbb{E}\left[\exp\left( \frac{1}{2} \int_0^{\tau^{*}} \left(z_i + \frac{1}{2} \sin(h + Y_h)\right)^2 dh \right)\right]
&\leq \mathbb{E}\left[\exp\left( \frac{1}{2} \cdot (3.5)^2 \cdot \tau^{*} \right)\right],
\end{align*}
which is finite under integrability assumptions on $\tau^{*}$. Hence, Novikov’s condition is satisfied.\\
\noindent We now specify the functions $\gamma_1$, $\gamma_2$, and $A$ used in the rejection probability of the Exact method:
\begin{equation*}
\begin{aligned}
\gamma_1(t) &= z \\[0.4em]
\gamma_2(t, (y, z)) &= \frac{1}{2}\sin(t+y) + \frac{1}{2}\left(\frac{1}{2}\cos(t+y)+(\zeta_i + \frac{1}{2}\sin(t+y))^{2}\right), \\[0.4em]
A(t, (y, z)) &= z y -\frac{1}{2}\cos(t+y).
\end{aligned}    
\end{equation*}
\noindent It is straightforward to verify that all three functions are smooth and bounded from above on intervals preceding the threshold crossing time, ensuring the applicability of the Exact algorithm.\\
\noindent We apply Algorithm~\ref{alg:final algorithm} to simulate samples of $\min(\tau_{\beta}, T_f)$ for the PDifMP described above. To validate the accuracy of our method, we also simulate the same quantity using the high-resolution Euler–Maruyama method. The resulting densities are approximated and displayed in Figure~\ref{fig:example2}.

\begin{figure}[H]
    \centering
    \includegraphics[width=0.6\linewidth]{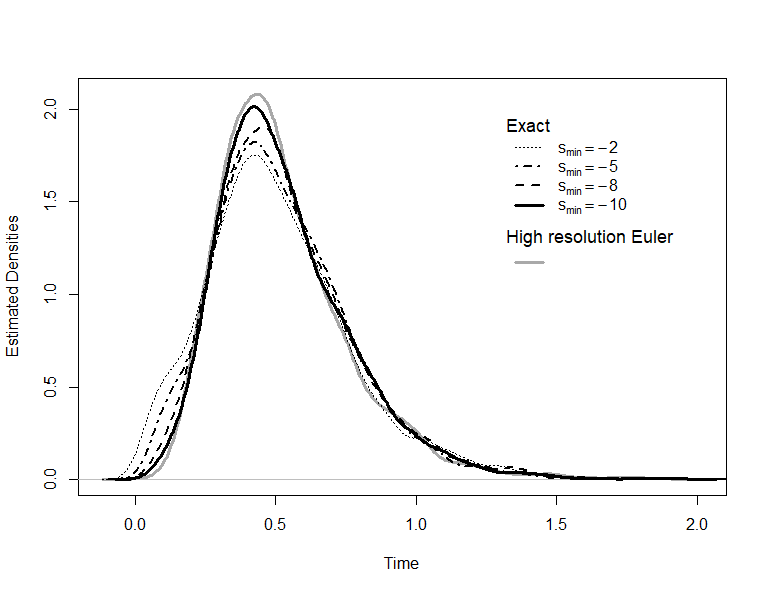}
    \caption{FPT densities for the PDifMP model including jump effects in both drift and diffusion, based on $3 \cdot 10^3$ simulations. Results are compared for (i) the Exact method with varying values of $s_\text{min}$, and (ii) the high-resolution Euler–Maruyama method. The jump rate is set to $\lambda = 1$.}
    \label{fig:example2}
\end{figure}

\noindent Figure~\ref{fig:example2} demonstrates a close similarity between the density curves produced by the Exact and Euler–Maruyama methods, particularly as the parameter $s_\text{min}$ becomes more negative. To assess this similarity quantitatively, we perform the KS test again to compare the empirical distributions obtained by both methods.

\begin{table}[H]
\centering
\begin{tabular}{|c|c|c|c|c|}
\hline
$s_\text{min}$ & -2 & -5 & -8 & -10 \\
\hline
p-value  & 0.0000  & 0.0055  & 0.2124 & 0.2764 \\
\hline
\end{tabular}
\caption{KS-test p-values comparing Exact method and Euler–Maruyama method}
\label{tab:s1_pvalue}
\end{table}

As shown in Table~\ref{tab:s1_pvalue}, the p-values increase significantly as $s_\text{min}$ becomes more negative. For values $s_\text{min} \leq -8$, the Exact method produces FPT distributions that are close to those obtained by the Euler–Maruyama method, confirming the theoretical convergence established in Theorem~\ref{thm:valid_alg_2}.

\begin{remark}
In these numerical examples, we use a linear threshold $\beta(t) = -t + 1$. This choice allows us to clearly isolate and assess the effects of the algorithm’s parameters $s_\text{min}$ and $\epsilon$. In particular, the first-passage time of Brownian motion to this linear boundary has an explicit distribution, eliminating numerical errors when simulating the auxiliary stopping time $\tau^*$. For more general time-dependent thresholds, we refer to the analysis in \cite{khurana2024exact}.
\end{remark}

\section{Conclusion} \label{sec:outlook}
In the present work, we extend the Exact simulation method for FPT in two crucial directions (i) for PDifMPs and (ii) to time-dependent thresholds. We use the algorithm proposed in \cite{herrmann2023exact}, which applies Exact method for SDEs interval-wise to simulate FPT of PDifMPs. Following this idea, we construct the algorithm that requires two key random quantities: (1) the FPT of SDE to time-dependent thresholds, and (2) the value of the SDE at a fixed time, conditional that it remains below the threshold upto that time.\\
While the interval-wise simulation framework originates from \cite{herrmann2023exact}, we extend it to the setting of generalised PDifMPs and time-dependent thresholds. Our goal is to simulate both components using the Exact method. The second component, in particular, presents significant challenges in the time-varying threshold case. To simulate it using Exact method we need the Brownian motion satisfying the condition as an auxiliary process and the corresponding probability weight. To address this we (i) introduce a method for simulating a point from the path of such a Brownian motion, (ii) use a Bessel Bridge to generate a full path passing through the simulated point, (iii) derive the rejection probability corresponding to this path, and (iv) complete the sampling procedure accordingly.\\
We identify sources of bias in the method and provide a theoretical proof of convergence. To validate the approach, we apply it to two distinct examples and compare the results with those obtained from a high-resolution Euler–Maruyama method. The comparison, based on density plots and statistical tests, confirms the convergence of our algorithm.\\
This method is particularly suitable for applications where high accuracy is critical and conventional time-discretisation techniques are inadequate, for instance in the case of high order SDEs.
\section*{Acknowledgements}
Devika Khurana thanks the Linz Center of Mechatronics GmbH for the financial support for this work.

\begin{appendices}

\section{}
\subsection{Proof of Theorem \ref{thm:girsanov_stopping} }
\label{Appendix_A}
\begin{proof}
We divide the proof into three steps: (i) showing that $(M_t)_{t \geq 0}$ is a true $\mathbb{P}$-martingale up to time $\mathcal{T}$; 
(ii) verifying that under $\mathbb{Q}$, the process $\tilde{B}_t := B_t + \int_0^t \alpha\big(s,\, (X^{i,\infty}_h, z_i)\big)\, dh$ is a Brownian motion; and  (iii) establishing the expectation transformation identity.

\paragraph{\textit{Step 1.}} For $z_i$ fixed in $E_2$, recall the exponential process $M_t$
\begin{equation*}
 M_t := \exp\left( - \int_0^t \alpha\big(h,\,(X^{i,\infty}_h, z_i)\big)\, dB_h - \frac{1}{2} \int_0^t \alpha^2(h,\, (X^{i,\infty}_h, z_i)\big)\, dh \right), \quad t \geq 0.   
\end{equation*}
\noindent Using Itô’s formula, we have:
\begin{equation*}
dM_t = -M_t \alpha\big(t,\, (X^{i,\infty}_t, z_i)\big)\, dB_t.    
\end{equation*}
\noindent Since $\mathcal{T}$ is a stopping time, the stopped process $(M_{t \wedge \mathcal{T}})_{t \geq 0}$ is a local martingale. 
By the Novikov condition \eqref{novik}, it is uniformly integrable, hence a true martingale, see \cite{oksendal2013stochastic} for more details. 
Therefore, the measure $\mathbb{Q}$ defined via
\begin{equation*}
\left. \frac{d\mathbb{Q}}{d\mathbb{P}} \right|_{\mathcal{F}_{\mathcal{T}}} = M_{\mathcal{T}}    
\end{equation*}
is well-defined and satisfies $\mathbb{Q} \sim \mathbb{P}$ on $\mathcal{F}_{\mathcal{T}}$.
\paragraph{\textit{Step 2.}} Under $\mathbb{Q}$, define the process
\begin{equation*}
 \tilde{B}_t := B_t + \int_0^t \alpha\big(h,\, (X^{i,\infty}_h, z_i)\big)\, dh, \quad 0 \leq t \leq \mathcal{T}.   
\end{equation*}
By Girsanov’s theorem (see e.g. \cite{oksendal2013stochastic}), for fixed $z_i$ values, $(\tilde{B}_t)_{0 \leq t \leq \mathcal{T}}$ is a standard Brownian motion under $\mathbb{Q}$. Hence, the original SDE under $\mathbb{P}$:
\begin{equation*}
 dX^{i,\infty}_t = \alpha\big(t,\, ( X^{i,\infty}_t, z_i)\big)\, dt + dB_t   
\end{equation*}
\noindent becomes, under $\mathbb{Q}$,
\begin{equation*}
dX^{i,\infty}_t = d\tilde{B}_t, \quad 0 \leq t \leq \mathcal{T}. \end{equation*}
This shows that $(X^{i,\infty}_t)_{0 \leq t \leq \mathcal{T}}$ behaves as a Brownian motion under $\mathbb{Q}$.

\paragraph{\textit{Step 3.}} Let $\Psi: C([0, \mathcal{T}]; \mathbb{R}) \to \mathbb{R}$ be a bounded, measurable functional. 
Using the definition of the Radon–Nikodym derivative \eqref{measure_transf}, we obtain:
\begin{equation*}
\mathbb{E}_{\mathbb{P}}\left[\Psi(X^{i,\infty}_h; 0 \leq h \leq \mathcal{T})\right]
= \mathbb{E}_{\mathbb{Q}}\left[\Psi(X^{i,\infty}_h; 0 \leq h \leq \mathcal{T}) M_{\mathcal{T}} \right].    
\end{equation*}
\noindent Since $M_{\mathcal{T}} > 0$ $\mathbb{P}$-almost surely and $\mathbb{E}_{\mathbb{P}}[M_{\mathcal{T}}] = 1$, 
the measures $\mathbb{Q}$ and $\mathbb{P}$ are equivalent on $\mathcal{F}_{\mathcal{T}}$. Note that absolute continuity may not hold beyond $\mathcal{T}$.
\end{proof}

\subsection{Proof of Theorem \ref{thm:Girsanovprob_FPT} }
\label{Appendix_B}

\begin{proof}
This proof follows the methodology of \cite{khurana2024exact}, with appropriate modifications tailored to our setting. We use Girsanov's transformation described in Theorem \ref{thm:girsanov_stopping}, resulting in the following equality:
\footnotesize{
\begin{equation}
    \mathbb{E}_{\mathbb{P}}\left[\Psi(\tau_{\beta}^{\text{cont},i})\mathbf{1}_{\{\tau_{\beta}^{\text{cont},i}<\infty\}}\right] 
    = \mathbb{E}_{\mathbb{Q}}\left[\Psi(\tau_{\beta}^{\text{cont},i}) \exp \left( \int_{T_{i}}^{\tau_{\beta}^{\text{cont},i}} \alpha\big(t,\,( {\tilde{X}^{i}_t},z_i)\big) dB_t - \frac{1}{2} \int_{T_{i}}^{\tau_{\beta}^{\text{cont},i}} \alpha^2\big(t,\,( {\tilde{X}_t^{i}}, z_i)\big) dt \right) \right].
    \label{eq:proof1}
\end{equation}}
\normalsize
\noindent To simplify the stochastic integral, we apply It\^o’s lemma to the function $A(t, (x,z))$ defined in Subsection \ref{subsection3.2}. Since $\frac{\partial}{\partial x}A(t, (x,z)) = \alpha(t, (x,z))$, we have
\begin{multline}
        \int_{T_{i}}^{\tau_{\beta}^{\text{cont},i}} \alpha\big(t,\,( {\tilde{X}_t^{i}}, z_i)\big) dB_t = A(\tau_{\beta}^{\text{cont},i}, \,({\tilde{X}_{\tau_{\beta}^{\text{cont},i}}^{i}},z_i)) - A(T_{i},\,( {\tilde{X}^{i}_{T_{i}}},z_i)) -\\ \int_{T_{i}}^{\tau_{\beta}^{\text{cont},i}} \left( \frac{\partial}{\partial t}A\big(t,\,( {\tilde{X}_t^{i}}, z_i)\big) + \frac{1}{2} \frac{\partial}{\partial x}\alpha\big(t,\,( {\tilde{X}_t^{i}}, z_i)\big) \right) dt.
\label{eq:proof2}
\end{multline}

\noindent Substituting \eqref{eq:proof2} into \eqref{eq:proof1}, we obtain
\small{
\begin{align}
\mathbb{E}_{\mathbb{P}}\left[\Psi(\tau_{\beta}^{\text{cont},i})\mathbf{1}_{\{\tau_{\beta}^{\text{cont},i}<\infty\}}\right] 
    &= \mathbb{E}_{\mathbb{Q}} \Bigg[ \Psi(\tau_{\beta}^{\text{cont},i}) 
    \exp \Bigg( A(\tau_{\beta}^{\text{cont},i},\,( \tilde{X}^i_{\tau_{\beta}^{\text{cont},i}},z_i)) - A(T_{i},\,( {\tilde{X}^{i}_{T_{i}}},z_i)) \notag \\
    &\quad - \int_{T_{i}}^{\tau_{\beta}^{\text{cont},i}} \left( \frac{\partial}{\partial t}A\big(t,\,( {\tilde{X}_t^{i}}, z_i)\big) + \frac{1}{2} \frac{\partial}{\partial x}\alpha\big(t,\,( {\tilde{X}_t^{i}}, z_i)\big) + \frac{1}{2} \alpha^2\big(t,\,( {\tilde{X}_t^{i}}, z_i)\big) \right) dt \Bigg) \Bigg].
\label{eq:proof3}
\end{align}}
\normalsize
\noindent Under the measure $\mathbb{Q}$, $\tau_{\beta}^{\text{cont},i}$ is the FPT of the process ${\tilde{X}^{i}}$, thus we have
\begin{equation*}
      \tilde{X}^i_{\tau_{\beta}^{\text{cont},i}} = \beta(\tau_{\beta}^{\text{cont},i}).
  \end{equation*}
\noindent Substituting this into \eqref{eq:proof3} gives
\footnotesize{
\begin{align*}
\mathbb{E}_{\mathbb{P}}\left[\Psi(\tau_{\beta}^{\text{cont},i})\mathbf{1}_{\{\tau_{\beta}^{\text{cont},i}<\infty\}}\right] 
    &= \mathbb{E}_{\mathbb{Q}} \Bigg[ \Psi(\tau_{\beta}^{\text{cont},i}) e^{ - A(T_{i},\,( {\tilde{X}^{i}_{T_{i}}},z_i))}
    \exp \Bigg( A(\tau_{\beta}^{\text{cont},i}, \, (\beta(\tau_{\beta}^{\text{cont},i}),z_i)) \notag \\
    &\quad - \int_{T_{i}}^{\tau_{\beta}^{\text{cont},i}} \left( \frac{\partial}{\partial t}A\big(t,\,( {\tilde{X}_t^{i}}, z_i)\big) + \frac{1}{2} \left( \frac{\partial}{\partial x}\alpha(t,\, ( {\tilde{X}_t^{i,\infty}}, z_i)) + \alpha^2\big(t,\,( {\tilde{X}_t^{i}}, z_i)\big) \right) \right) dt \Bigg) \Bigg].
\end{align*}}
\normalsize
\noindent Assumption~\ref{ass: on drift} allows us to write the following:
\begin{equation}
A(\tau_{\beta}^{\text{cont},i}, \,(\beta(\tau_{\beta}^{\text{cont},i}),z_i))=\int_{T_{i}}^{\tau_{\beta}^{\text{cont},i}} \left(\frac{\partial}{\partial t}A(t,\,( \beta(t),z_i)) + \alpha(t,\,( \beta(t), z_i))\beta'(t)\right) dt.
\label{eq: proof4}
\end{equation}
\noindent Rewriting the integral terms, using  equation \eqref{eq: proof4} and the pre-defined functions $\gamma_1$ and $\gamma_2$, we obtain
\begin{equation}
    \mathbb{E}_{\mathbb{P}}\left[\Psi(\tau_{\beta}^{\text{cont},i})\mathbf{1}_{\{\tau_{\beta}^{\text{cont},i}<\infty\}}\right] 
    = e^{-A(T_{i},\,(  {\tilde{X}^{i}_{T_{i}}},z_i)) + A(T_{i},\,( \beta(T_{i}),z_i))} \mathbb{E}_{\mathbb{Q}} \left[\Psi(\tau_{\beta}^{\text{cont},i}) \eta^{[2]}(\tau_{\beta}^{\text{cont},i}) \right],
\label{eq:proof_final}
\end{equation}
where $\eta^{[2]}(t)$ is defined in \eqref{eq:eta_def}.
\end{proof}
\end{appendices}

\bibliographystyle{plain}
\bibliography{Ref}
\end{document}